\documentclass[11pt]{amsart}

 \usepackage{amsfonts,amssymb,amsmath,mathrsfs}

\usepackage{color,latexsym,amsfonts}

\numberwithin{equation}{section}

 \def\qed{\hfill$\Box$\medskip}
 \newtheorem{theorem}{Theorem}[section]
 \newtheorem{lemma}[theorem]{Lemma}

 \newtheorem{proposition}[theorem]{Proposition}

 \def\<{\langle}\def\>{\rangle}

 \def\beqlb{\begin{eqnarray}}\def\eeqlb{\end{eqnarray}}
 \def\beqnn{\begin{eqnarray*}}\def\eeqnn{\end{eqnarray*}}
  \def\ar{\!\!&}

\begin{document}

\centerline{\LARGE\bf On Seneta-Heyde Scaling for a stable branching}

\medskip
\centerline{\LARGE\bf random walk \footnote{This work is supported by NSFC
  (No.  11371061, 11531001, 11671041).}}
\bigskip\bigskip

\centerline{Hui He,\,Jingning Liu\,\footnote{Corresponding author.  } and Mei Zhang\,\footnote{Corresponding author.}}

\bigskip\bigskip

{\narrower{\narrower

\noindent{\it Abstract.} We consider a discrete-time branching random walk in the boundary case, where the associated random walk is in the domain of attraction of an $\alpha$-stable law with $1<\alpha<2$. We prove that the derivative martingale $D_n$ converges to a non-trivial limit $D_\infty$ under some regular conditions. We also study the additive martingale  $W_n$, and prove $n^\frac{1}{\alpha}W_n$ converges in probability to a constant multiple of $D_\infty$.

\bigskip


\bigskip


\par}\par}

\bigskip\bigskip

\section{Introduction}

We consider a discrete-time one-dimensional branching random walk. It starts with an initial ancestor particle located at the origin. At time $1$, the particle dies, producing a certain number of new particles. These new particles are positioned according to the distribution of the point process $\Theta$. At time $2$, these particles die, each giving birth to new particles positioned (with respect to the birth place) according to the law of $\Theta$. And the process goes on with the same mechanism. We assume the particles produce new particles independently of each other at the same generation and of everything up to that generation. This system can be seen as a branching tree $\textbf{T}$ with the origin as the root.

For each vertex $x$ on $\textbf{T}$, we denote its position by $V(x)$. The family of the random variables $(V(x))$ is usually referred as a branching random walk (Biggins~\cite{bi}).

Throughout the paper, we assume  the boundary case (in the sense of \cite{BK}):
\begin{align}\label{C:1.1}
\mathbf{E}\Big(\sum_{|x|=1}1\Big)>1, \;\; \mathbf{E}\Big(\sum_{|x|=1}e^{-V(x)}\Big)=1,\;\; \mathbf{E}\Big(\sum_{|x|=1}V(x)e^{-V(x)}\Big)=0,
\end{align}
where $|x|$ denotes the generation of $x$. Every branching random walk satisfying certain mild integrability assumptions can be reduced to this case by some renormalization; see Jaffuel \cite{j} for more details. Note that \eqref{C:1.1} implies  $\textbf{T} $ is a super-critical Galton-Watson tree.

One could immediately see from  \eqref{C:1.1} that
\begin{align}
W_n:=\sum_{|x|=n}e^{-V(x)},\quad n\geq0, \nonumber
\end{align}
is a martingale, which is referred to as the \emph{additive martingale} in the literature. Since $(W_n)$ is nonnegative, it converges almost surely to $0$ (see Biggins \cite{bi2} and Lyons~\cite{Ly}).
It is natural to ask at which rate $W_n$ goes to $0$, and this is our main issue to discuss.

This issue is usually called Seneta-Heyde norming problem. As for the Seneta-Heyde theorem for Galton-Watson
processes, see Heyde \cite{hey}, Seneta \cite{se}, etc. The study of the Seneta-Heyde norming for the branching random walk in a general case (i.e., without assumption \eqref{C:1.1}) goes back at least to Biggins and Kyprianou \cite{BK3} and \cite{BK4}. The boundary case has first been considered by Biggins and Kyprianou \cite{BK}. Later, the converging rate of $W_n$ associated with the one-dimensional random walk with finite variance (see (\ref{C:1.2})) has been investigated by Hu and Shi~\cite{hs}.
Recently,  Aidekon and Shi \cite{AS1} established the exact rate for $W_n$ under some  weaker integrability assumptions.

Define
\begin{align}
D_n:=\sum_{|x|=n}V(x)e^{-V(x)},\quad n\geq0. \nonumber
\end{align}

Under the assumption $\mathbf{E}\big(\Sigma_{|x|=1}V(x)e^{-V(x)}\big)=0$, one can easily check that $(D_n)$
is also a martingale, which is referred  as the \emph{derivative martingale} associated with $(V(x))$. It is the basis of our discussion on additive martingale $W_n$. The convergence of derivative martingale and related questions have been extensively studied. One can see Barral \cite{Ba}, Biggins \cite{bi4} and \cite{bi5} for non-boundary cases, Kyprianou \cite{Ky} and Liu \cite{Liu} for the boundary case. Later, Biggins and Kyprianou \cite{BK2} studied it by considering the branching random walks as multi-type branching processes and gave a Kesten-Stigum  theorem for the mean convergence.
However, there is a small ``gap" between the necessary condition and the sufficient condition.
Recently, Aidekon \cite{AI} and Chen \cite{Chen} filled the ``gap". To state their results, we give the following  integrability condition:
\begin{align}
\;\mathbf{E}\Big(\sum_{|x|=1}V(x)^2e^{-V(x)}\Big)<\infty.\label{C:1.2}
\end{align}

\noindent{\bf Theorem A} (Biggins and Kyprianou \cite{BK2}).
 Assume \eqref{C:1.1} and \eqref{C:1.2}. Then there exists a  nonnegative random
variable $D_\infty$ such that
\begin{align}
D_n\rightarrow D_\infty, \quad\mathbf{P}\text{-}a.s. \nonumber
\end{align}

\noindent{\bf Theorem B} (Aidekon \cite{AI}, Chen \cite{Chen}).
Assume \eqref{C:1.1} and \eqref{C:1.2}. Then
$\mathbf{P}(D_\infty>0)>0$  if and only if the following condition holds:
\begin{align}
\;\mathbf{E}\big(X\log_{+}^{2}{X}+\widetilde{X}\log_{+} {\widetilde{X}}\big)<\infty, \label{C:1.3}
\end{align}
where $\log_{+}y:=\max\{0,\log y\}$, $\log^2_+y:=(\log_+y)^2$ for any $y>0$, and
\begin{align}
X:=\sum_{|x|=1}e^{-V(x)},\quad \widetilde{X}:=\sum_{|x|=1}V(x)_+e^{-V(x)}, \label{E:1.5}
\end{align}
with $V(x)_+:=\max\{V(x),0\}$. Moreover, when $D_n$ is non-trivial, $\mathbf{P}(D_\infty=0)$  equals to the extinction probability of the branching random walk.

Many discussions in this paper are trivial if $\bf T$ is finite. So let us introduce the conditional probability
\begin{align}
\mathbf{P}^*(\cdot):=\mathbf{P}(\cdot\,|\mbox{non-extinction}). \nonumber
\end{align}
Obviously $W_n\rightarrow0,\; \mathbf{P}^*\text{-}a.s.$

\noindent{\bf Theorem C} (Aidekon and Shi \cite{AS1}).
Assume \eqref{C:1.1}, \eqref{C:1.2} and \eqref{C:1.3}. Under $\mathbf{P}^*$, we have
\begin{align}
\lim_{n\rightarrow\infty} n^{1/2}W_n={\big(\frac{2}{\pi\sigma^2}\big)}^{1/2}D_\infty, \quad \text{in probability}, \nonumber
\end{align}
where $D_\infty>0$ is the random variable in Theorem A, and
\begin{align}
\sigma^2:=\mathbf{E}\big(\sum_{|x|=1}V(x)^2e^{-V(x)}\big)<\infty. \nonumber
\end{align}

In this  paper, instead of  \eqref{C:1.2} and \eqref{C:1.3}, we shall study $W_n$ under the following assumptions:
\begin{align}
&(\romannumeral1)\;\;\mathbf{E}\,\Big(\sum_{|x|=1}\mathbf{1}_{\{V(x)\leq -y\}} e^{-V(x)}\Big)=o(y^{-\alpha}),\quad y\to \infty; \label{stable1}\\
&(\romannumeral2)\;\;\mathbf{E}\,\Big(\sum_{|x|=1}\mathbf{1}_{\{V(x)\geq y\}} e^{-V(x)}\Big)\sim \frac{c}{y^{\alpha}},\quad y\to \infty; \label{stable2}\\
&(\romannumeral3)\;\;\mathbf{E}\big(X(\log_{+}{X})^{\alpha}+\widetilde{X}(\log_{+}{\widetilde{X}})^{\alpha-1}\big)<\infty,  \label{stable3}
\end{align}
where $\alpha\in (1,2)$, $c>0$, $X$ and $\widetilde{X}$ are defined by \eqref{E:1.5}. Under these assumptions,  the step of the one-dimensional random walk associated with $(V(x))$ (see Section 2) belongs to the domain of attraction of an $\alpha$-stable law. We remark here that for some technical reasons, constant $c$ in (\ref{stable2}) can not be replaced by a general slowly varying function.

\smallskip

We are ready to state our main results.

\begin{theorem}\label{T:1.1}
Assume \eqref{C:1.1}, \eqref{stable1}, \eqref{stable2}. Then there exists a finite nonnegative random
variable $D_\infty$ such that
\begin{align}
D_n\rightarrow D_\infty, \;\;\mathbf{P}\text{-}a.s. \label{E:1.1}
\end{align}
Moreover, if condition \eqref{stable3} holds additionally, then  $\mathbf{P}^*(D_\infty>0)=1$.
\end{theorem}

\begin{theorem}\label{T:1.3}
Assume \eqref{C:1.1}, \eqref{stable1}, \eqref{stable2} and \eqref{stable3}. We have, under $\mathbf{P}^*$,
\begin{align}
\lim_{n\rightarrow\infty}n^{\frac{1}{\alpha}}W_n=\frac{\theta}{\Gamma(1-1/\alpha)}D_\infty
\quad\text{in probability}, \nonumber
\end{align}
where $D_\infty$ is given in Theorem~\ref{T:1.1}, and $\theta$ is a positive constant defined in (\ref{E:21}).
\end{theorem}


\begin{theorem}\label{T:1.4}
Assume \eqref{C:1.1}, \eqref{stable1}, \eqref{stable2} and \eqref{stable3}. We have,
\begin{align}
\limsup_{n\rightarrow\infty} \;n^{\frac{1}{\alpha}}W_n=\infty\;\;\;\mathbf{P}^*\!\!-\!\!a.s. \nonumber
\end{align}
\end{theorem}

\noindent{\bf Remark.} Theorem \ref{T:1.4} tells us that the convergence in probability in Theorem \ref{T:1.3} is optimal, which it can not be strengthened to almost surely convergence.

\section{Stable random walk}

In this section,  we first introduce an one-dimensional random walk associated with the branching random walk. Then we give some properties which are essential in the proofs of Theorems \ref{C:1.1}, \ref{C:1.2} and \ref{C:1.3}.

Throughout, for any vertex $x$, $x_i$ ($0\leq i\leq |x|$) denotes the ancestor of $x$ at the $i$-th generation (in particular, $x_0=\varnothing, x_{|x|}=x$).

\noindent{\bf The many-to-one formula.}
For $a\in\mathbb{R}$, we denote by $\mathbf{P}_a$ the probability distribution associated to the branching random walk $(V(x))$ starting from $a$, and $\mathbf{E}_a$ the corresponding expectation. Under \eqref{C:1.1}, there exists a  sequence of independently and identically distributed (i.i.d.) real-valued random variables $S_1,S_2-S_1,S_3-S_2,\ldots,$ such that for any $n\geq1, a\in\mathbb{R}$ and any measurable function $g:\mathbb{R}^n\rightarrow[0,\infty),$
\begin{align}
\mathbf{E}_a\Big(\sum_{|x|=n}g\big(V(x_1),\ldots,V(x_n)\big)\Big)=\mathbf{E}_a\Big(e^{S_n-a}g(S_1,\ldots,S_n)\Big), \nonumber
\end{align}
where, under $\mathbf{P}_a$, we have $S_0=a$ almost surely. We will write $\mathbf{P}$ and $\mathbf{E}$ instead of $\mathbf{P}_0$ and $\mathbf{E}_0$. Since $\mathbf{E}\big(\sum_{|x|=1}V(x)e^{-V(x)}\big)=0$, we have $\mathbf{E}(S_1)=0$. 
Under conditions \eqref{stable1} and \eqref{stable2},  $S_1$ belongs to the domain of attraction of a spectrally positive stable law with characteristic function
\begin{align}
G_{\alpha,-1}(t):=\exp\big\{-c_0|t|^\alpha\big(1-i\frac{t}{|t|}\tan{\frac{\pi\alpha}{2}}\big)\big\},\quad c_0>0. \nonumber
\end{align}
We denote by $S_1\in\mathcal{D}(\alpha,-1)$.

Next we recall some elementary properties of $(S_n)$ from existed literatures.

The (strict) \emph{descending ladder heights} of $(S_n)$ are $\,Z_0\!>\!Z_1\!>\!Z_2\!>\!\ldots\,$, if
$Z_k:=S_{\tau^-_k},$ with $\tau^-_0:=0$ and $\tau^-_k:=\inf\{i>\tau^-_{k-1}:S_i<\min_{0\leq j\leq \tau^-_{k-1}}S_j\}$, $k\geq1$. Let
\begin{align}
R(u):=\sum_{k=0}^\infty \mathbf{P}(|Z_k|\leq u). \nonumber
\end{align}
We know $Z_1,Z_2 - Z_1,Z_3 - Z_2,\ldots$ are also i.i.d. random variables. By Bingham \cite{bing}, $\mathbf{E}(Z_1)<\infty$. So the renewal theorem holds as follows
\begin{align}
\theta:=\lim_{u\rightarrow\infty} \frac{R(u)}{u}. \label{E:21}
\end{align}
Consequently, there exist constants $c_2\geq c_1>0$ such that
\begin{align}
c_1(1+u)\leq R(u) \leq c_2(1+u),  \quad u\geq0. \label{E:22}
\end{align}
In the following,  $c_3, c_4, c_5,... $ are positive constants.

Similarly we introduce the (strict) \emph{ascending ladder heights} $H_0\!<\!H_1\!<\!H_2\!<\!\ldots$, if
$H_k:=S_{T_k},$ with $T_0:=0$ and $T_k:=\inf\{i>T_{k-1}:S_i>\max_{0\leq j\leq T_{k-1}}S_j\}$, $k\geq1$, and define:
\begin{align}
K(u):=\sum_{k=0}^\infty \mathbf{P}\{|H_k|\leq u\}. \label{E:23}
\end{align}
Sinai \cite{Sinai} and Rogozin \cite{Rogo} proved that
\begin{align} \label{E:24}
\mathbf{P}(H_1\geq x)\sim \frac{c_3}{x^{\alpha-1}}, \quad x \rightarrow +\infty.
\end{align}
By Feller \cite[Chap. XIV, (3.4)]{FELL},
\begin{align}
K(x) \sim \frac{1}{\Gamma(1-\alpha\rho)\Gamma(1+\alpha\rho)}\cdot\frac{1}{\mathbf{P}(H_1\geq x)},\quad  x\rightarrow \infty. \nonumber
\end{align}
Substituting \eqref{E:24} into this equation yields
\begin{align}
K(x) \sim c_4x^{\alpha-1},\quad  x\rightarrow \infty. \label{E:25}
\end{align}
 As a consequence, there exist constants $c_6\geq c_5>0$ such that
\begin{align}
c_5(1+x)^{\alpha - 1}\leq K(x) \leq c_6(1+x)^{\alpha - 1},  \;\;x\geq0. \label{E:22b}
\end{align}

\begin{lemma}(Bingham \cite{bing2}).
We have for $x\ge 0$,
\begin{align}
&\mathbf{P}(\underline{S}_n\geq -x) \sim \frac{R(x)}{\,n^{\frac{1}{\alpha}}\Gamma(1-\frac{1}{\alpha})}\quad n\rightarrow\infty;\label{E:26}\\
&\mathbf{P}(\underline{-S}_n\geq -x) \sim \frac{K(x)}{\,n^{1-\frac{1}{\alpha}}\Gamma(\frac{1}{\alpha})}\quad n\rightarrow\infty,\label{E:27}
\end{align}
where $\underline{S}_n:=\min_{i\leq n}S_i$ and $\underline{-S}_n:=-\max_{i\leq n}S_i$.
\end{lemma}

\begin{lemma}\label{L:2.2}
There exists $c_7>0$ such that for $a\geq0, b\geq-a$ and $n\geq1$,
\begin{align}
\mathbf{P}\big\{b\leq S_n\leq b+1, \underline{S}_n\geq-a\big\} \leq c_7\,\cdot\,\frac{\;(1+a)(1+a+b)^{\alpha-1}}{n^{1+\frac{1}{\alpha}}}.
\end{align}
\end{lemma}

\begin{proof}
The proof is similar to Aidekon and Shi \cite{AS2}.
We only prove the case of  $n=3k$, $k\ge 1$. A similar argument works for the cases of $n=3k+1$  and $n=3k+2$.

According to Stone's local limit theorem,  there exist $c_8>0$ and $c_9>0$ such that $\forall\; h\geq c_8$ and $n\geq 1$,
\begin{align}\label{E:29}
\sup_{r\in\mathbb{R}}\mathbf{P}(r\leq \pm S_n\leq r+h)\leq c_9\cdot\frac{h}{n^{\frac{1}{\alpha}}}.
\end{align}
By the Markov property at time $k$, we have
\begin{align}
&\mathbf{P}\big\{b\leq S_{3k}\leq b+c_8, \underline{S}_{3k}\geq -a\big) \nonumber\\
&\quad\leq\mathbf{P}(\underline{S}_k\geq -a)\sup_{x\geq-a}\mathbf{P}(b-x\leq S_{2k}\leq b-x+c_8, \underline{S}_{2k}\geq -a-x) .
\end{align}
Let $\widetilde{S}_j:=S_{2k-j}-S_{2k}$. Then \beqnn&&\mathbf{P}(b-x\leq S_{2k}\leq b-x+c_8, \underline{S}_{2k}\geq -a-x)\cr&&\quad\leq \mathbf{P}(-b+x-c_8\leq \widetilde{S}_{2k}\leq -b+x, \;\min_{1\leq i\leq 2k} \widetilde{S}_i\geq -a-b-c_8). \eeqnn Applying
 the Markov property, for $x\geq-a$,
\begin{align}
&\mathbf{P}(b-x\leq S_{2k}\leq b-x+c_8, \underline{S}_{2k}\geq -a-x)\nonumber\\
&\quad\leq\mathbf{P}(\min_{1\leq i\leq k}\widetilde{S}_i\geq -a-b-c_8)\sup_{y\in\mathbb{R}}\mathbf{P}(-b+x-c_8-y\leq \widetilde{S}_k\leq -b+x-y).
\end{align}
Then \eqref{E:22} and \eqref{E:25}, together with \eqref{E:26},  \eqref{E:27}  and \eqref{E:29}, yield the Lemma.
\qed
\end{proof}

\begin{lemma}\label{L:2.3}
There exists a constant $c_{10}>0$ such that for any $b\geq-a$ and $n\geq1$,
\begin{align}
\mathbf{P}(S_n\leq b,\,\underline{S}_n\geq-a)\leq c_{10}\cdot \frac{\;(1+a)(1+a+b)^\alpha}{n^{1+\frac{1}{\alpha}}}.\nonumber
\end{align}
\end{lemma}

\begin{proof}
It is immediate from Lemma \ref{L:2.2}.\qed
\end{proof}

The following result is also an extension of Aidekon \cite[Lemma B.2]{AI}.

\begin{lemma}\label{L:2.4}
There exists a constant $c_{11}>0$ such that for any $z\geq0$ and $x>0$,
\begin{align}
\sum_{l\geq0}\mathbf{P}_z(S_l\leq x,\,\underline{S}_l\geq0)\leq c_{11}\,(1+x)^{\alpha-1}(1+\min(x,z)).\nonumber
\end{align}
\end{lemma}

\begin{proof}
 We first consider $x<z$. Define $\tau_x=\inf\{n,S_n\leq x\}$. Then we have
\begin{align}
\sum_{l\geq0}\mathbf{P}_z(S_l\leq x,\underline{S}_l\geq0)&=\mathbf{E}_z\Big(\sum_{l\geq\tau_x}\mathbf{1}_{\{S_l\leq x,\underline{S}_l\geq0\}}\Big)\nonumber\\
&\leq \mathbf{E}\Big(\sum_{l\geq0}\mathbf{1}_{\{S_l\leq x,\underline{S}_l\geq -x\}}\Big),\nonumber
\end{align}
where we used the Markov property at time $\tau_x$. We obtain, by Lemma \ref{L:2.3},
\begin{align}\label{E:2.12}
\mathbf{E}\Big(\sum_{l\geq0}\mathbf{1}_{\{S_l\leq x,\underline{S}_l\geq -x\}}\Big)&\leq 1+x^\alpha+\sum_{l> x^\alpha}\mathbf{P}(S_l\leq x,\underline{S}_l\geq -x)\nonumber\\
&\leq 1+x^\alpha+c_{10} 2^\alpha\cdot\sum_{l>x^\alpha}\frac{(1+x)^{\alpha+1}}{l^{1+\frac{1}{\alpha}}}\nonumber\\
&\leq c_{12}\cdot(1+x)^\alpha.
\end{align}
Next, we consider $x\geq z$. Then
\begin{align}
\sum_{l\geq0}\mathbf{P}_z(S_l\leq x,\underline{S}_l\geq0)\leq\sum_{l\leq x^\alpha}\mathbf{P}_z(\underline{S}_l\geq0)+\sum_{l>x^\alpha}\mathbf{P}_z(S_l\leq x,\underline{S}_l\geq0).\nonumber
\end{align}
 \eqref{E:26} implies $\mathbf{P}_z(\underline{S}_l\geq0)\leq c_{13}\cdot\frac{(1+z)}{l^\frac{1}{\alpha}}$ for any $l\geq0$. Again, using Lemma \ref{L:2.3} gives
\begin{align}
\mathbf{P}_z(S_l\leq x,\underline{S}_l\geq0)\leq c_{10}\cdot \frac{(1+z)(1+x)^\alpha}{l^{1+\frac{1}{\alpha}}}.\nonumber
\end{align}
Thus
\begin{align}
\sum_{l\geq0}\mathbf{P}_z(S_l\leq x,\underline{S}_l\geq0)&\leq c_{13}\cdot\sum_{l\leq x^\alpha}\frac{(1+z)}{l^\frac{1}{\alpha}}+c_{8}\sum_{l>x^\alpha}\frac{(1+z)(1+x)^\alpha}{l^{1+\frac{1}{\alpha}}}\nonumber\\
&\leq c_{14}(1+z)(1+x)^{\alpha-1},
\end{align}
which, together with \eqref{E:2.12}, completes the proof.
\qed

\end{proof}

\begin{lemma} \label{L:2.5}
Let $\{d_n\}$ be a sequence of positive numbers such that $d_n=o(n^{1/\alpha})$. Then for any bounded continuous function $f$, we have
\begin{align}
E\Big(f\left(\frac{S_n+x}{n^{1/\alpha}}\right)\mathbf{1}_{\{\underline{S}_n\geq-x\}}\Big)=\frac{R(x)}{\,\Gamma(1-\frac{1}{\alpha})n^{1/\alpha}}\Big(\int_0^\infty f(t)p_{\alpha}(t)dt+o_n(1)\Big), \label{E:212}
\end{align}
uniformly in $u\in[0,d_n]$, where $p_{\alpha}$ is the density function of a nonnegative random variable $M_{\alpha}$. And $M_{\alpha}$ satisfies
\begin{align}
\lim_{n\rightarrow\infty}\mathbf{P}\Big(\frac{S_n}{n^{1/\alpha}}\in[u_1,u_2)\Big|\underline{S}_n>0\Big)=\mathbf{P}\big(M_\alpha\in[u_1,u_2)\big). \label{E:213}
\end{align}
Moreover, \begin{align}
\mathbf{E}\,(M_\alpha)=\frac{\Gamma(1-\frac{1}{\alpha})}{\theta}.
\label{EMa}
\end{align}
\end{lemma}

\begin{proof}
The proof of \eqref{E:212} is an extension of the Lemma 2.2 of Aidekon and Jaffuel \cite{AJ}, where the random walk has finite variance. So we omit it here.  \eqref{E:213} can be found in Vatutin and Wachtel \cite[Theorem 1]{VW}. In the following we shall prove (\ref{EMa}). Denote by $p_\alpha$ and $g_\alpha$ the density functions of $M_\alpha$ and $S_1$, respectively. By (79) in \cite{VW},
\begin{align}
p_\alpha(z)=\int_0^1\frac{t^{-1/\alpha}dt}{(1-t)^{1/\alpha}}
\int_0^{\frac{z}{t^{1/\alpha}}}g_\alpha\bigg(\frac{z-t^{1/\alpha}u}{(1-t)^{1/\alpha}}\bigg)p_\alpha(u)du.
\nonumber
\end{align}
 Hence,
\beqlb
\mathbf{E}(M_\alpha)\ar=\ar
\int_0^\infty zp_\alpha(z)dz\cr&=&\int^\infty_0z\int_0^1\frac{t^{-1/\alpha}dt}{(1-t)^{1/\alpha}}\int_0^{\frac{z}{t^{1/\alpha}}}g_\alpha
\bigg(\frac{z-t^{1/\alpha}u}{(1-t)^{1/\alpha}}\bigg)p_\alpha(u)dudz  \nonumber\\
\ar=\ar\int_0^1\frac{t^{-1/\alpha}dt}{(1-t)^{1/\alpha}}\int^\infty_0p_\alpha(u)du\int^\infty_{t^{1/\alpha}u} g_\alpha\bigg(\frac{z-t^{1/\alpha}u}{(1-t)^{1/\alpha}}\bigg)zdz  \nonumber\\
\ar=\ar \int_0^1 t^{-1/\alpha} dt \int_0^\infty p_\alpha(u)du\int_0^{\infty}g_{\alpha}(z)\left((1-t)^{1/\alpha}z+t^{1/\alpha}u\right)dz\nonumber\\
\ar=\ar\int_0^1t^{-1/\alpha}(1-t)^{1/\alpha}dt\int_0^\infty g_\alpha(z)zdz+\int_0^\infty up_\alpha(u)du\int_0^\infty g_\alpha(z)dz\nonumber\\
\ar=\ar B(1-\frac{1}{\alpha}, 1+\frac{1}{\alpha})\int_0^\infty g_\alpha(z)zdz+\mathbf{E}(M_\alpha) \int_0^\infty g_\alpha(z)dz,\label{Lem2.5a}
\eeqlb
where $B$ is the Beta function. By the properties of stable law, $1-\frac{1}{\alpha}=\int_0^\infty g_\alpha(z)dz$; see (19) in \cite{VW}. Then (\ref{Lem2.5a}) implies
\begin{align}\label{Ma}
\mathbf{E}\,(M_\alpha)=\Gamma(1-\frac{1}{\alpha})\Gamma(\frac{1}{\alpha})\int_0^\infty g_\alpha(z)zdz.
\end{align}
By Bingham \cite{bing}, we have
\begin{align}
\mathbf{E}(e^{-tS_1};S_1>0)=\frac{1}{\alpha}\sum_{n=0}^\infty\frac{(-t\,\theta)^n}{\Gamma(1+n\alpha^{-1})}.\nonumber
\end{align}
Taking derivatives of $t$ and  letting $t=0$, we arrive at
\begin{align}
\int_0^\infty g_\alpha(z)zdz=\mathbf{E}\big(S_1\mathbf{1}_{\{S_1\geq0\}}\big)=\frac{1}{\Gamma(\frac{1}{\alpha})\theta}.\nonumber
\end{align}
Combining with (\ref{Ma}), we complete the proof. \qed
\end{proof}

The following lemma is a preparation  for Lemma~\ref{L:2.6}. Its proof is  similar to that for the case $\alpha=2$ (see Aidekon and Shi \cite{AS1}), so we omit it.
\begin{lemma}\label{L:2.55}
There exists $c_{15}>0$ such that for $a\geq0$,
\begin{align}
\sup_{n\geq1}\mathbf{E}\,\big(|S_n|\mathbf{1}_{\{\underline{S}_n\geq-a\}}\big)\leq c_{15}(1+a).\nonumber
\end{align}
\end{lemma}

\begin{lemma}\label{L:2.6}
Let $0<\lambda<1$. There exists a constant $c_{16}>0$ such that for $a,b\geq0$, $0\leq u\leq v$ and $n\geq1$,
\begin{align}
& \mathbf{P}\big(\underline{S}_{\;\llcorner\lambda n\lrcorner}\geq-a,\min_{i\in[\lambda n,n]\cap \mathbb{Z}}S_i\geq b-a,S_n\in[b-a+u,b-a+v]\,\big)\nonumber\\
&\leq  c_{16}\cdot\frac{(1+v)^{\alpha\!-\!1}(1+v-u)(1+a)}{n^{1+\frac{1}{\alpha}}}. \label{E:2.1}
\end{align}
\end{lemma}

\begin{proof}
The idea of the proof is borrowed from Aidekon and Shi \cite{AS1}. Without  loss of generality, we treat $\lambda n$ as an integer. Let $\mathbf{P}_{\eqref{E:2.1}}$ be the probability expression of the left-hand side of \eqref{E:2.1}. By the Markov property at time $\lambda n$, we have
\begin{align}
\mathbf{P}_{\eqref{E:2.1}}=\mathbf{E}\Big(\mathbf{1}_{\{\underline{S}_{\lambda n}\geq -a,\,S_{\lambda n\geq b-a}\}}f(S_{\lambda n})\Big),\nonumber
\end{align}
where $f(r)=\mathbf{P}\Big(\underline{S}_{n-\lambda n}\geq b-a-r, S_{n-\lambda n}\in[b-a-r+u,b-a-r+v]\Big)$ for $r\geq b-a$. Then it follows from Lemma \ref{L:2.2} that
\begin{align}
f(r)\leq c_7\cdot \frac{\;(1+r-b+a)(1+v)^{\alpha-1}(v-u+1)}{(1-\lambda)^{1+\frac{1}{\alpha}}n^{1+\frac{1}{\alpha}}} \;,\;\;\;r\geq b-a.\nonumber
\end{align}
Therefore,
\begin{align}
\mathbf{P}_{\eqref{E:2.1}}\leq c_7\cdot \frac{\;(1+v)^{\alpha-1}(v-u+1)}{(1-\lambda)^{1+\frac{1}{\alpha}}n^{1+\frac{1}{\alpha}}}\mathbf{E}\big((S_{\lambda n}+a-b+1)\mathbf{1}_{\{\underline{S}_{\lambda n}\geq-a,\,S_{\lambda n}\geq b-a\}}\big).\nonumber
\end{align}
The expectation above on the right-hand side is bounded by $\mathbf{E}\big(\,\big|S_{\lambda n}\big|\mathbf{1}_{\{\underline{S}_{\{\lambda n\}}\geq-a\}}\big)+a+1$. This together with Lemma \ref{L:2.55} leads to \eqref{E:2.1}.
\end{proof}\qed

The next result is an analog of the classical Stone local limit theorem. One can find the proof in Vatutin and Wachtel \cite[Theorem 3 and Theorem 5]{VW}.
\begin{lemma} \label{L:2.65}
There exists a constant $c_{17}>0$ such that for any $0<a\leq b<\infty$, we have
\begin{align}
\liminf_{n\rightarrow\infty}n^{\frac{1}{\alpha}}\inf_{u\in[\,an^{1/2},\,bn^{1/2}]}\mathbf{P}\big(u\leq S_n<u+c_{17}\big|\underline{S}_n\geq0\big)>0.\nonumber
\end{align}
\end{lemma}

\begin{lemma}\label{L:2.7}
There exists $c_{18}>0$ such that for any sequence $(a_n)$ of nonnegative numbers with $\limsup_{\,n\rightarrow\infty}\frac{a_n}{\,n^{\frac{1}{\alpha}}}<\infty$,
\begin{align}
\liminf_{\,n\rightarrow\infty} n^{1+\frac{1}{\alpha}}\mathbf{P}\big(\underline{S}_n\geq0,\min_{n<j\leq2n}S_j\geq a_n,a_n\leq S_{2n}\leq a_n+c_{18}\big)>0. \nonumber
\end{align}
\end{lemma}

The proof can be found in Aidekon and Shi \cite{AS2} Lemma 4.3 for $\alpha=2$. The same proof is valid for $1<\alpha<2$, since we have an analog of Stone local limit theorem (i.e. Lemma \ref{L:2.65}).

  Define $\tilde{S_i}=-S_i$, $i\ge 0$. Next we will give some analogues of Lemmas \ref{L:2.2}, \ref{L:2.4}, \ref{L:2.55} and \ref{L:2.6} for the new random walk $\{\tilde{S_i}\}$.  We have,

\begin{lemma}\label{L:2.10}
 There exists $c_{19},c_{20}>0$ such that for $a\geq0,b>-a$ and $n\geq1$,
\begin{align}
&(\romannumeral1) \;\;\;\; \mathbf{P}\big\{b\leq \tilde{S}_n\leq b+1, \underline{\tilde{S}}_n\geq-a\big\} \leq c_{19}\,\cdot\,\frac{\;(1+a)^{\alpha-1}(1+a+b)}{n^{1+\frac{1}{\alpha}}},\\
&(\romannumeral2) \;\;\; \sum_{i\geq0}\mathbf{P}\big\{\underline{\tilde{S}}_i\geq-a,\tilde{S_i}\leq b\}\leq c_{20}(1+a+b)(1+a)^{\alpha-1}.
\end{align}
\end{lemma}
The proof of this lemma is very close to Lemma \ref{L:2.2} and Lemma \ref{L:2.4}, so we omit it here.

\begin{lemma}\label{L:2.55b}
There exists $c_{21}>0$ such that for $a\geq0$,
\begin{align}
\sup_{n\geq1}\textbf{E}\,\big(|\tilde{S}_n|^{\alpha-1}\textbf{1}_{\{\underline{\tilde{S}}_n\geq-a\}}\big)\leq c_{21}(a+1)^{\alpha-1}.\nonumber
\end{align}
\end{lemma}

\begin{proof}
In the following we denote the renewal function of $(\tilde{S}_n)$ by $R_{\tilde{S}}(x)$. It is easy to see that  $R_{\tilde S}(x)=K(x)$.
By \cite[Lemma 3.1]{CC}, we have \begin{align*}   \sup_{n\geq1}\textbf{E}_a\,\big(R_{\tilde S}(\tilde{S}_n) \textbf{1}_{\{\underline{\tilde{S}}_n\geq 0\}}\big)=R_{\tilde S}(a)=K(a)\le c_6(1+a)^{\alpha-1}.
\end{align*}
By (\ref{E:25}) and (\ref{E:22b}),  $R_{\tilde S}(\tilde{S}_n)\textbf{1}_{\{\underline{\tilde{S}}_n\geq 0\}}= R_{\tilde{ S}}(|\tilde{S}_n|)\textbf{1}_{\{\underline{\tilde{S}}_n\geq 0\}}\ge c_5  |\tilde{S}_n|^{\alpha-1}\textbf{1}_{\{\underline{\tilde{S}}_n\geq 0\}}$. Combining above discussions we get
\begin{align*}
&  \sup_{n\geq1}\textbf{E}\,\big(|\tilde{S}_n|^{\alpha-1}\textbf{1}_{\{\underline{\tilde{S}}_n\geq-a,\tilde{S}_n>0\}}\big)\\
&=  \sup_{n\geq1}\textbf{E}_a\,\big(|\tilde{S}_n-a|^{\alpha-1}\textbf{1}_{\{\underline{\tilde{S}}_n\geq 0,\tilde{S}_n>a\}}\big)\\
&\le  a^{\alpha-1}+\sup_{n\geq1}\textbf{E}_a\,\big(|\tilde{S}_n|^{\alpha-1}\textbf{1}_{\{\underline{\tilde{S}}_n\geq 0\}}\big)\\
&\le  (1+\frac{c_6}{c_5})(1+a)^{\alpha-1}.\end{align*}
Clearly, \begin{align*}
\sup_{n\geq1}\textbf{E}\,\big(|\tilde{S}_n|^{\alpha-1}\textbf{1}_{\{\underline{\tilde{S}}_n\geq-a,\tilde{S}_n\le 0\}}\big)\le\sup_{n\geq1}\textbf{E}\,\big((-\tilde{S}_n)^{\alpha-1}\textbf{1}_{\{ -a\le \tilde{S}_n\le 0\}}\big) \le a^{\alpha-1}.\end{align*}
The proof is completed.

\end{proof}\qed

\begin{lemma}\label{L:2.6b}
Let $0<\lambda<1$, $\rho<0$ and $-\rho=o(n^{\frac{1}{\alpha}})$. There exists a constant $c_{22}>0$ such that for $a\geq0$, $0\leq u\leq v$ and $n\geq1$,
\begin{align}
& \textbf{P}\{\underline{\tilde{S}}_{\;\llcorner\lambda n\lrcorner}\geq-a,\min_{i\in[\lambda n,n]\cap \mathbb{Z}}(\tilde{S_i})\geq \rho-a,\tilde{S}_n\in[\rho-a+u,\rho-a+v]\,\}\nonumber\\
&\leq  c_{22}\cdot\frac{(1+v) (v-u+1)(1+a)^{\alpha - 1}}{n^{1+\frac{1}{\alpha}}}. \label{E:2.1}
\end{align}
\end{lemma}
\begin{proof}
The proof is essentially similar to Lemma \ref{L:2.6} except that  applying Lemma \ref{L:2.55b} in the last step.
\end{proof}\qed

\section{Truncated martingale and spine decomposition}

In order to study the convergence of $D_n$ and $W_n$, it turns out more convenient to work with a truncated version of the branching random walk. Let $(V(x))$ be a branching random walk satisfying \eqref{C:1.1}. For any vertex $x$, we denote the unique shortest path relating $x$ to the root $\varnothing$ by $\langle\varnothing,x\rangle$ . Define
\begin{align}
\underline{V}(x):=\min_{y\in\langle\varnothing,x\rangle}V(y). \nonumber
\end{align}
For $\beta\geq0,$ denote
\begin{align}
R_\beta(u):=R(u+\beta),\;\;u\geq-\beta. \nonumber
\end{align}
Now we use the renewal function $R_\beta(\cdot)$ to introduce the truncated processes
\begin{align}
&W_n^{\beta}:=\sum_{|x|=n}e^{-V(x)}\mathbf{1}_{\{\underline{V}(x)\geq -\beta\}}, \nonumber\\
&D_n^{\beta}:=\sum_{|x|=n}R_\beta(V(x))e^{-V(x)}\mathbf{1}_{\{\underline{V}(x)\geq -\beta\}}. \nonumber
\end{align}

Since $W_n\rightarrow0, \;\,\mathbf{P}\!-\!a.s.$, we have $\min_{|x|=n}V(x)\rightarrow\infty$ almost surely on the set of non-extinction. By (\ref{E:21}), $\lim_{u\rightarrow\infty} \frac{R(u)}{u}=\theta$. It is evident that if $\beta$ is sufficiently large, then on the set of non-extinction, $W_n^{\beta}$  behaves like $W_n$ and $D_n^{\beta}$ behaves like $\theta D_n$.

Under assumption \eqref{C:1.1}, for any $\beta\geq0$, the process $(D_n^{\beta}, n\geq0)$ is a nonnegative martingale with respect to natural filtration $(\mathcal{F}_n)$ (see Aidekon \cite{AI}). For all $n$, it follows from Kolmogorov's extension theorem that there exists a unique probability measure ${\hat{\mathbf{P}}}^\beta_a$ on $\mathcal{F}_\infty$ $(\mathcal{F}_\infty:=\vee_n {\mathcal{F}_n})$ such that
\begin{align}
\frac{\;d\,{\hat{\mathbf{P}}}^\beta_a \;}{\;d\,\mathbf{P}_a\;}\Bigg|_{\mathcal{F}_n}=
\frac{D_n^\beta}{\;R_\beta(a)e^{-a}}, \quad n\geq1. \label{E:1.4}
\end{align}
Notice that ${\hat{\mathbf{P}}}^\beta_a(\text{non-extincion})=1$, we can say under ${\hat{\mathbf{P}}}^\beta_a$ the process will never extinct almost surely.

We introduce a point process $\hat{\Theta}$ whose distribution is the law of $(V(x),|x|=1)$ under ${\hat{\mathbf{P}}}^\beta(:={\hat{\mathbf{P}}}^\beta_0)$. Consider the following process. At time $0$, there is one particle $\omega_0^\beta$ located at $V(\omega_0^\beta)=a$. At each step $n$, particles at generation $n$ die, and produce independently new particles point according to the law of $\Theta$, except one particle denoted by $\omega_n^\beta$ which generates particles according to $\hat{\Theta}$ (with respect to the birth position $V(\omega_n^\beta)$). The particle $\omega_{n+1}^\beta$ is chosen among the children $y$ of $\omega_{n}^\beta$ with probability proportional to $R_\beta(V(y))e^{-V(y)}\mathbf{1}_{\{\underline{V}(y)\geq-\beta\}}$. This defines a branching random walk with a marked sequence $(\omega_{n}^\beta)$, which we call the \emph{spine}. We denote the above system by $\mathcal{B}^\beta_a$.

\begin{theorem}\label{T:3.1}  (Biggins and Kyprianou \cite{BK2}). Assume \eqref{C:1.1} and $\beta\geq0$.\\
(\romannumeral 1) The branching random walk $(V(x))$ under ${\hat{\mathbf{P}}}^\beta_a$, is distributed as $\mathcal{B}^\beta_a$.\\
(\romannumeral 2) For any $n$ and any vertex $x$ with $|x|=n$, we have
\begin{align}
{\hat{\mathbf{P}}}^\beta_a (\omega_{n}^\beta=u|\mathcal{F}_n)=\frac{\;R_\beta(V(u))e^{-V(u)}\mathbf{1}_{\{\underline{V}(u)\geq-\beta\,\}}}{D_n^\beta}.
\end{align}
(\romannumeral 3) The spine process $(V(\omega_{n}^\beta), n\geq0)$ under ${\hat{\mathbf{P}}}^\beta_a$ is distributed as the random walk $(S_n)_{n\geq0}$ conditioned to stay above $-\beta$ under $\mathbf{P}_a$. More precisely,  for any $n\geq1$ and any measurable function $g: \mathbb{R}^{n}\rightarrow [0,\infty),$
\begin{align}
&\hat{\mathbf{E}}^\beta_a\big(g(V(\omega_i^\beta)), 0\leq i\leq n\big)\nonumber\\
&\;\;\;=\frac{1}{R_\beta(a)}\mathbf{E}_a\big(g(S_i,0\leq i\leq n)R_\beta(S_n),\mathbf{1}_{\{\underline{S}_n\geq -\beta\}}\big) .\label{E:3.2}
\end{align}
\end{theorem}

\section{Proof of Theorem \ref{T:1.1}}

In this section, we will divide our proof into two parts.
In part 1, we will prove \eqref{E:1.1}; In part 2, we will prove that the limit is strictly positive, $\mathbf{P}^*\text{-}a.s.$

\noindent{\bf Part 1.} Let $\Omega_\beta:=\{V(x)>-\beta,\, \forall\; n\geq0,|x|=n\,\}\cap\{$\,nonextinction$\}$. Since $\min_{|x|=n}V(x)\rightarrow \infty,\; \mathbf{P}^*-a.s.$ holds under  \eqref{C:1.1}, we have $\mathbf{P}^*(\Omega_\beta)\to 1$ as $\beta\rightarrow\infty$. Fixing $\beta\geq0$, we only need consider the event in $\Omega_\beta$.  By \eqref{E:21}, For any $\;\varepsilon>0$, there exists $N $ such that for any $n>N, |x|=n$,
\begin{align}
(\frac{1}{\theta}-\varepsilon)R_\beta(V(x))<V(x)+\beta< (\frac{1}{\theta}+\varepsilon)R_\beta(V(x)) .\nonumber
\end{align}
Then
\begin{align}
\sum_{|x|=n}(\frac{1}{\theta}-\varepsilon)R_\beta(V(x))e^{-V(x)}<\sum_{|x|=n}V(x)e^{-V(x)}+\beta\sum_{|x|=n}e^{-V(x)}<\sum_{|x|=n}(\frac{1}{\theta} + \varepsilon)R_\beta(V(x))e^{-V(x)}.\nonumber
\end{align}
i.e.,
\begin{align}
(\frac{1}{\theta}-\varepsilon)D_n^\beta<D_n+\beta W_n<(\frac{1}{\theta}+\varepsilon)D_n^\beta \label{E:4.1}
\end{align}
holds in the set $\Omega_\beta$.
We denote the limit of $D_n^\beta$ by $D_\infty^\beta$. Taking $n\rightarrow\infty$ on both sides of \eqref{E:4.1} and then letting   $\varepsilon\to 0$,
we obtain $\lim_{n\to \infty} D_n =\frac{1}{\theta} D^\beta_\infty$ on the set $\Omega_\beta$. Now letting $\beta\rightarrow\infty$, we complete the proof of this part.

\noindent{\bf Part 2.}
We borrow an idea from Aidekon \cite[Proposition A.3]{AI}.
\begin{align}
D_n&=\sum_{|x|=n}V(x)e^{-V(x)}\nonumber \\&=\sum_{|y|=1}\sum_{|x|=n,\,x\geq y}V(x)e^{-V(x)}\nonumber \\
&=\sum_{|y|=1}e^{-V(y)}\sum_{|x|=n,x\geq y}(V(x)-V(y))e^{-(V(x)-V(y))}\nonumber\\
&\;\;\;\;\;+\sum_{|y|=1}V(y)e^{-V(y)}\sum_{|x|=n,\,x\geq y}e^{-(V(x)-V(y))}\nonumber\\
&=\sum_{|y|=1}e^{-V(y)}D_{n-1,\,y}+\sum_{|y|=1}V(y)e^{-V(y)}W_{n-1,\,y} \nonumber
\end{align}
where $D_{n-1,\,y}, W_{n-1,\,y}$ are the corresponding functions on the subtree rooted at $y$. Apparently
$D_{n-1,\,y}\stackrel{d}{=}D_{n-1}$, $W_{n-1,\,y}\stackrel{d}{=}W_{n-1}$. Letting $n\rightarrow\infty$, we get
\begin{align}
D_\infty=\sum_{|y|=1}e^{-V(y)}D_{\infty,\,y}, \;\;\mathbf{P}^*-a.s. \;(\;also\; \mathbf{P}-a.s.)\nonumber
\end{align}
where $D_{\infty,\,y}\stackrel{d}{=}D_\infty$. Let $p:=\mathbf{P}(D_\infty=0)$. Since $(D_{\infty,\,y},|y|=1)$ are independent of each other and $V(y)$, the equation  $p=\mathbf{E}(p^{\sum_{|y|=1}1})$  holds. As a consequence of the branching property, we immediately have $p\!=\!1$ or $p\!=\!\mathbf{P}$(extinction)=$:q$. We only need to prove $\mathbf{P}(D_\infty>0)>0$.
Clearly,
\begin{align}
D_n^\beta&=\sum_{|x|=n}R_\beta(V(x))e^{-V(x)}\mathbf{1}_{\{\underline{V}(x)\geq-\beta\}}\nonumber\\
&\leq c_2\sum_{|x|=n}\big(1+\beta+V(x)_+\big)e^{-V(x)}\nonumber\\
&\leq c_{23}\bigg(W_n+\sum_{|x|=n}V(x)_+e^{-V(x)}\bigg). \nonumber
\end{align}
Let $n\rightarrow\infty$ gives $D_\infty^\beta\leq c_{23}D_\infty\;\;\mathbf{P}^*-a.s.$ (also $\mathbf{P}-a.s.$ by $D_\infty^\beta=D_\infty=0$ on the set \{extinction\}).
Hence we only need to prove $\mathbf{P}(D^\beta_\infty>0)>0$. Noticing that $\mathbf{E} D_0^\beta=R(\beta)>0$, it is sufficient to prove that
$(D_n^\beta, n\geq0)$ is uniformly integrable under $\mathbf{P}$.

For any $\omega_i^\beta$ on the spine, we define
\begin{align}
\Omega(\omega_i^\beta):=\big\{\,|x|=i:x>\omega_{i-1}^\beta,x\neq\omega_i^\beta\big\}.\nonumber
\end{align}
In a word, $\Omega(\omega_i^\beta)$ stands for the set of all ``brothers" of $\omega_i^\beta$. Let
\begin{align}
\hat{\mathcal{G}}^\beta_\infty:=\sigma\{\omega^\beta_j,V(\omega^\beta_j),\Omega(\omega^\beta_j),(V(u))_{u\,\in\,\Omega(\omega^\beta_j)},j\geq1\}\nonumber
\end{align}
be the $\sigma$-algebra of the spine and its brothers. Using the martingale property of $D^\beta_n$ for the subtrees rooted at brothers of the spine, we have
\begin{align}
\hat{\mathbf{E}}^\beta\big(D^\beta_n\,|\,\hat{\mathcal{G}}^\beta_\infty\big)=&R_\beta(V(\omega^\beta_n))e^{-V(\omega^\beta_n)}\nonumber\\
&\;\;+\sum_{k=1}^{n}\sum_{x\in\,\Omega(\omega^\beta_k)} R_\beta(V(x))e^{-V(x)}\mathbf{1}_{\{\underline{V}(x)\geq-\beta\}}.\nonumber
\end{align}
By Lemma~\ref{L:4.3} in Section 4, $V(\omega^\beta_n)\rightarrow\infty,\; \hat{\mathbf{P}}^\beta\!-a.s.$ Therefore $R_\beta(V(\omega^\beta_n))e^{-V(\omega^\beta_n)}$ goes to zero as $n\rightarrow\infty$. We already know that $R(x)\leq c_2(1+x)_+\leq c_2(1+x_+)$ for any $x\in \mathbb{R}$. Then, by Fatou's lemma,
\begin{align}
\hat{\mathbf{E}}^\beta\big(D^\beta_\infty\,|\,\hat{\mathcal{G}}^\beta_\infty\big)&\leq \liminf_{n\rightarrow\infty}\hat{\mathbf{E}}^\beta\big(D^\beta_n\,|\,\hat{\mathcal{G}}^\beta_\infty\big)\nonumber\\
&\leq c_2\sum_{k\geq1}\sum_{x\,\in\,\Omega(\omega^\beta_k)}\big(1+\big(\beta+V(x)\big)_+\big)e^{-V(x)}\nonumber\\
&\leq c_2(A_1+A_2), \nonumber
\end{align}
with
\begin{align}\label{A1}
A_1:=\sum_{k\geq1}\big(1+\beta+V(\omega^\beta_{k-1})\big)e^{-V(\omega^\beta_{k-1})}\sum_{x\in\Omega(\omega^\beta_k)}e^{-(V(x)-V(\omega^\beta_{k-1}))},
\end{align}
\begin{align}\label{A2}
A_2:=\sum_{k\geq1}e^{-V(\omega^\beta_{k-1})}\sum_{x\in\Omega(\omega^\beta_k)}
{\big(V(x)-V(\omega^\beta_{k-1})\big)}_+e^{-(V(x)-V(\omega^\beta_{k-1}))}.
\end{align}
We shall show $A_1$ and $A_2$ are finite. Let us consider $A_1$ first. Recall the definition of $X:=\sum_{|x|=1}e^{-V(x)}$, $\widetilde{X}:=\sum_{|x|=1}V(x)_+e^{-V(x)}$. Define $X':=\sum_{|x|=1}R_\beta(V(x))e^{-V(x)}\mathbf{1}_{\{V(x)\geq-\beta\}}$. Clearly, for any $a\geq-\beta$,
\begin{align}
X'\leq c_2\sum_{|x|=1}e^{-V(x)}\big((1+a+\beta)+(V(x)-a)_+\big). \nonumber
\end{align}
Therefore, we have for any $z\geq0$ and $a\geq-\beta$,
\begin{align}
\hat{\mathbf{P}}_a^\beta&\Big(\sum_{|x|=1}e^{-(V(x)-a)}>z\Big)\nonumber\\
&=\frac{1}{R_\beta(a)e^{-a}}\mathbf{E}_a\bigg(X'\mathbf{1}_{\{\sum_{|x|=1}e^{-(V(x)-a)}>z\}}\bigg)\nonumber\\
&\leq c_{24}\mathbf{E}_a\Bigg(\sum_{|x|=1}e^{-(V(x)-a)}\bigg(1+\frac{(V(x)-a)_+}{1+a+\beta}\bigg)\mathbf{1}_{\{\sum_{|x|=1}e^{-(V(x)-a)}>z\}}\Bigg)\nonumber\\
&=c_{24}\mathbf{E}\big(X\mathbf{1}_{\{X>z\}}\big)+\frac{c_{24}}{1+a+\beta}\mathbf{E}\big(\widetilde{X}\mathbf{1}_{\{X>z\}}\big)\nonumber\\
&=:c_{24}h_1(z)+\frac{c_{24}}{1+a+\beta}h_2(z).\nonumber
\end{align}
We deduce by the Markov property at time $k-1$ that
\begin{align}
\hat{\mathbf{P}}^\beta &\Big(\sum_{x\in\Omega(\omega^\beta_k)}e^{-(V(x)-V(\omega^\beta_{k-1}))}\geq e^{V(\omega^\beta_{k-1})/2}\Big)\nonumber\\
&\leq c_{24}\, \hat{\mathbf{E}}^\beta \Big(h_1(e^{V(\omega^\beta_{k-1})/2})+\frac{1}{1+V(\omega^\beta_{k-1})+\beta}h_2(e^{V(\omega^\beta_{k-1})/2})\Big).\nonumber
\end{align}
Hence,
\begin{align}
&\sum_{k\geq1}\hat{\mathbf{P}}^\beta\Big(\sum_{x\in\Omega(\omega^\beta_k)}e^{-(V(x)-V(\omega^\beta_{k-1}))}\geq e^{V(\omega^\beta_{k-1})/2}\Big)\nonumber\\
&\;\;\;\;\leq c_{24} \sum_{l\geq0}\hat{\mathbf{E}}^\beta\big(h_1(e^{V(\omega^\beta_l)/2})\big)\nonumber\\
&\;\;+ c_{24} \sum_{l\geq0}\hat{\mathbf{E}}^\beta\bigg(\frac{1}{1+V(\omega^\beta_l)+\beta}h_2(e^{V(\omega^\beta_l)/2})\bigg).\nonumber
\end{align}
We now estimate $\sum_{l\geq0}\hat{\mathbf{E}}^\beta\big(h_1(e^{V(\omega^\beta_l)/2})\big)$. By \eqref{E:3.2}, we have
\begin{align}
\hat{\mathbf{E}}^\beta\big(h_1(e^{V(\omega^\beta_l)/2})\big)&=\frac{1}{R(\beta)}\mathbf{E}\big(R_\beta(S_l)h_1(e^{S_l/2})
\mathbf{1}_{\{\underline{S}_l\geq-\beta\}}\big)\nonumber\\
&=\frac{1}{R(\beta)}\mathbf{E}\big(R_\beta(S_l)X\mathbf{1}_{\{S_l<2\ln X\}\bigcap\{\underline{S}_l\geq-\beta\}}\big), \nonumber
\end{align}
where we can choose the random walk $(S_n,n\geq0)$ independent of $X$. By Lemma 2.4,
\begin{align}
&\sum_{l\geq0}\hat{\mathbf{E}}^\beta \big(h_1(e^{V(\omega^\beta_l)/2})\big) \nonumber\\
&\leq \frac{1}{R(\beta)}\mathbf{E}\big(XR_\beta(2\ln X)\sum_{l\geq0}\mathbf{1}_{\{S_l<2\ln_+X,\, \underline{S}_l\geq-\beta\}}\big)  \nonumber\\
&\leq c_{25}\cdot \mathbf{E}\big(X(1+\ln_+X)^{\alpha}\big), \nonumber
\end{align}
which is finite by our assumption \eqref{stable3}. Similarly,
\begin{align}
\hat{\mathbf{E}}^\beta \Big(\frac{1}{1+V(\omega^\beta_l)+\beta}\cdot h_2(e^{V(\omega^\beta_l)/2})\Big)
\leq c_{26}\mathbf{E}\bigg(\widetilde{X}\mathbf{1}_{\{S_l\leq2\ln X, \underline{S}_l\geq-\beta\}}\bigg). \nonumber
\end{align}
Lemma \ref{L:2.4} and Lemma \ref{L:4.2} yield that
\begin{align}
\sum_{l\geq0}\hat{\mathbf{E}}^\beta\Big(\frac{1}{1+V(\omega^\beta_l)+\beta}\cdot h_2(e^{V(\omega^\beta_l)/2})\Big)\leq
c_{27} \mathbf{E}\big(\widetilde{X}(1+\ln_{+}X)^{\alpha-1}\big)<\infty .\nonumber
\end{align}
 Hence we have
\begin{align}
\sum_{k\geq1}\hat{\mathbf{P}}^\beta &\Big(\sum_{x\in\Omega(\omega^\beta_k)}e^{-(V(x)-V(\omega^\beta_{k-1}))}\geq e^{v(\omega^\beta_{k-1})}/2\Big)<\infty. \nonumber
\end{align}
By the Borel-Cantelli lemma,  there exists $K_0$ such that for  $k\ge K_0$,
\begin{align}
(1+\beta&+V(\omega^\beta_{k-1}))e^{-V(\omega^\beta_{k-1})}\sum_{x\in\Omega(\omega^\beta_k)}e^{-(V(x)-V(\omega^\beta_{k-1}))}\nonumber\\
&\leq(1+\beta+V(\omega^\beta_{k-1}))e^{-V(\omega^\beta_{k-1})/2}.  \nonumber
\end{align}
 According to Lemma \ref{L:4.3}, there exists $\epsilon\in(0,1-\frac{1}{\alpha})$ such that $V(\omega_k^\beta)\geq k^\varepsilon$, $\hat{\mathbf{P}}^\beta-a.s.$ for $k$ large enough. Combining with \eqref{A1}, we get $A_1\!<\!\infty$. Similarly we can prove $A_2\!<\!\infty$. Hence, $\hat{\mathbf{E}}^\beta\big(D^\beta_\infty|\hat{\mathcal{G}}^\beta_\infty\big)<\infty$, that implies
\begin{align}
\sup_{n\geq0}D^\beta_n<\infty,\quad \hat{\mathbf{P}}^\beta-a.s. \label{E:4.4}
\end{align}
Since for any $n\geq0,\; u\geq0$
\begin{align}
\mathbf{E}(D^\beta_n;D^\beta_n\geq u)=R(\beta)\hat{\mathbf{P}}^\beta(D^\beta_n\geq u)
\leq R(\beta)\hat{\mathbf{P}}^\beta(\sup_{n\geq0}D^\beta_n\geq u), \nonumber
\end{align}
then
\begin{align}
\sup_{n\geq0}\mathbf{E}(D^\beta_n;D^\beta_n\geq u)\leq R(\beta)\hat{\mathbf{P}}^\beta(\sup_{n\geq0}D^\beta_n\geq u) \nonumber
\end{align}
which goes to $0$ as $u\rightarrow\infty$ by \eqref{E:4.4}. We prove that the sequence of random variable $(D_n^\beta, n\geq0)$ is uniformly integrable under $\mathbf{P}$. We are done.
\qed

The followsing lemma is an extension of the case $\alpha=2$  proved in Aidekon \cite{AI}.
\begin{lemma} \label{L:4.2}
Let $X$ and $\widetilde{X}$ be the nonnegative random variables defined as \eqref{E:1.5} such that condition \eqref{stable3} holds. Then we have
\begin{align}
\mathbf{E}\,(X(\ln_+\widetilde{X})^{\alpha})<\infty\,,\;\mathbf{E}\,(\widetilde{X}(\ln_+X)^{\alpha-1})<\infty.\label{E:4.2}
\end{align}
Moreover, as $z\rightarrow\infty,$
\begin{align}
\mathbf{E}\,\big(X\big(\ln_+(X+\widetilde{X})\big)^{\alpha}\min(\ln_+(X+\widetilde{X}),z)\big)=o(z),\label{E:4.3}\\
\mathbf{E}\,\big(\widetilde{X}\big(\ln_+(X+\widetilde{X})\big)^{\alpha-1}\min\big(\ln_+(X+\widetilde{X}),z\big)\big)=o(z) \label{E:4.5} .
\end{align}
\end{lemma}

\begin{proof}
 We first prove \eqref{E:4.2}. It is easy to check that for any $x,\widetilde{x}\geq0$,
\begin{align}
x(\ln_+\widetilde{x})^\alpha\leq4x(\ln_+x)^\alpha+
\frac{\alpha}{\alpha-1}\,\widetilde{x}(\ln_+\widetilde{x})^{\alpha-1} .\nonumber
\end{align}
It follows that
\begin{align}
\mathbf{E}(X(\ln_+\widetilde{X})^\alpha)\leq4\mathbf{E}(X(\ln_+X)^\alpha)+\frac{\alpha}{\alpha-1}\mathbf{E}(
\,\widetilde{X}(\ln_+\widetilde{X})^{\alpha-1}),\nonumber
\end{align}
which is finite under condition \eqref{stable3}. Also, we observe that
\begin{align}
\widetilde{X}(\ln_+X)^{\alpha-1}\leq \max(\widetilde{X}(\ln_+\widetilde{X})^{\alpha-1},X(\ln_+X)^{\alpha-1})\nonumber
\end{align}
which implies that $\mathbf{E}(\widetilde{X}(\ln_+X)^{\alpha-1})<\infty$. We turn to prove \eqref{E:4.3}. For any $\varepsilon>0$ and all large $z$, we have
\begin{align}
\mathbf{E}&(X(\ln_+(X+\widetilde{X}))^\alpha\min(\ln_+(X+\widetilde{X}),z))\nonumber\\
&=\mathbf{E}(X(\ln_+(X+\widetilde{X}))^\alpha\min(\ln_+(X+\widetilde{X}),z),\ln_+(X+\widetilde{X})\geq\varepsilon z)\nonumber\\
&\;\;\;+\mathbf{E}(X(\ln_+(X+\widetilde{X}))^\alpha\min(\ln_+(X+\widetilde{X}),z),\ln_+(X+\widetilde{X})<\varepsilon z).\nonumber\\&\;\;\;\;\;\leq \varepsilon z(1+\mathbf{E}(X(\ln_+(X+\widetilde{X}))^\alpha)),\nonumber
\end{align}
where the last line follows from $\mathbf{E}(X(\ln_+(X+\widetilde{X}))^\alpha)<\infty$. Hence,
\begin{align}
\mathbf{E}\,\big(X\big(\ln_+(X+\widetilde{X})\big)^{\alpha}\min(\ln_+(X+\widetilde{X}),z)\big)=o(z).  \nonumber
\end{align}
Similarly, we can prove \eqref{E:4.5}.
\qed
\end{proof}

\begin{lemma}\label{L:4.3}
For any $\epsilon\in(0,1-\frac{1}{\alpha})$, we have
\begin{align}
\frac{\,V(\omega_k^\beta)}{k^\epsilon}\rightarrow \infty, \;\;\hat{\mathbf{P}}^\beta - a.s.\;\;\;k\rightarrow\infty. \nonumber
\end{align}
\end{lemma}
\begin{proof}
We only need to prove the case $\beta=0$. Here we introduce  Tanaka's construction for the random walk conditioned to stay positive (see Biggins \cite{bi3} and Tanaka \cite{tana}).
Recall that $\{S_n\}$ is the associated random walk with the branching random walk $(V(x))$, and $\tau$ is the first time the random walk $\{S_n\}$ hits $(0,\infty)$, that is,  the first strict ascending ladder time (also denoted by $T_1$). Hence, we obtain a random array $(S_j,0\leq j\leq \tau)$, and we denote it by $\xi=(\xi(j),0\leq j\leq \tau)$. Let $\{\xi_k=(\xi_k(j),0\leq j\leq \tau_k),k\geq1\}$ be a sequence of independent copies of $\xi$, where $\{\tau_k,k\geq1\}$ is a sequence of independent copies of $\tau$. For any $k\geq1$, let $w_k(j):=\xi_k(\tau_k)-\xi_k(\tau_k-j) $ for $0\leq j\leq \tau_k$. Recall that the successive strict ascending ladder times and heights for $\{S_n\}$ are $\{(T_k,H_k):k=0,1,2,\ldots\}$. So that $\{(T_n-T_{n-1},H_n-H_{n-1}),n\geq1\}$ are i.i.d. variables.


Define $\zeta_0=0$ and
\begin{align}
\zeta_n=H_k+w_{k+1}(n-T_k) ,\;\; T_k<n\leq T_{k+1}.\nonumber
\end{align}
Then $\{\zeta_n, n\ge 0\}$ is a construction for the random walks $(S_n)$ conditioned to stay positive.

From Vatutin and Wachtel \cite{VW}, we have
\begin{align}
\frac{T_k}{a_k}\stackrel{d}{\rightarrow}Y_1\quad \mbox{and}\quad \frac{H_k}{b_k}\stackrel{d}{\rightarrow}Y_{2},\,\;\;k\rightarrow\infty\nonumber
\end{align}
where $Y_1$, $Y_2$ are two random variables, and $\{(a_k)\}$ $\{(b_k)\}$ are specified respectively by $\mathbf{P}(T_1>a_k)\sim\frac{1}{k}$ and  $\mathbf{P}(H_1>b_k)\sim\frac{1}{k} $ as $k\rightarrow\infty$.
By Bingham \cite[Theorem 3]{bing3},
\begin{align}
\mathbf{P}(T_1>n)\sim\frac{1}{n^{1-\frac{1}{\alpha}}l(n)\Gamma(\frac{1}{\alpha})}\;\; as\;\; n\rightarrow\infty,\nonumber
\end{align}
for some function $l$ varying slowly at infinity. Then we immediately get $a_k^{1-\frac{1}{\alpha}}l(a_k)\Gamma(\frac{1}{\alpha})\sim k$, and similarly $b_k\sim k^{\frac{1}{1-\alpha}}$ by \eqref{E:24}. For the property of slowing varying function $l$, for $\varepsilon<1-\frac{1}{\alpha}$, we have
\begin{align}
\lim_{k\rightarrow\infty}\frac{a_k^\varepsilon}{k}=\lim_{k\rightarrow\infty}\frac{a_k^\varepsilon}{a_k^{1-\frac{1}{\alpha}}l(a_k)\Gamma(\frac{1}{\alpha})}=0.\nonumber
\end{align}
Then we get $\frac{T_k^\varepsilon}{k}\stackrel{d}{\rightarrow}0$. The convergence is also in probability. Hence, there exists a subsequence $\{n_k\}$ such that
\begin{align}
\frac{T_{n_k}^\varepsilon}{n_k}\rightarrow 0 \;\;\; \mathbf{P}-a.s.\nonumber
\end{align}
Then for $T_{n_k-1}<n\leq T_{n_k}$, by
\begin{align}
\zeta_n=H_{n_k-1}+w_{n_k\!}(n-T_{n_k-1}),\nonumber
\end{align}
 we obtain
\begin{align}
\frac{\zeta_n}{n^\varepsilon}\geq\frac{H_{n_k-1}}{{T_{n_k}}^\varepsilon}\geq \frac{H_{n_k-1}}{n_k}\cdot\frac{n_k}{{T_{n_k}}^\varepsilon}. \nonumber
\end{align}
By a generalization of the law of large numbers, $\frac{H_k}{k+1}\stackrel{a.s.}{\rightarrow}\infty$. Therefore
\begin{align}
\frac{\zeta_n}{n^\varepsilon}\rightarrow \infty, \quad n\rightarrow \infty\,,\;\;\mathbf{P}-a.s. \nonumber
\end{align}
From Theorem \ref{T:3.1} (\romannumeral 3), the spine process $(V(\omega_{n}^0), n\geq0)$ under ${\hat{\mathbf{P}}}$ is distributed as the random walk $(S_n)_{n\geq0}$ conditioned to stay positive under $\mathbf{P}$. The proof is completed.\qed
\end{proof}

\section{Convergence in probability of $\frac{W_n^\beta}{D_n^\beta}$ under $\mathbf{P}^\beta$}

In this section we do preparatory work for the proof of Theorem~\ref{T:1.1}. The idea is from Aidekon and Shi \cite{AS1}. The main result of this section is Proposition~\ref{P:5.1}. To prove it, we need auxiliary Lemmas \ref{L:5.3}--\ref{L:5.6}, which we put at the end of this section.

\begin{proposition}\label{P:5.1}
Assume \eqref{C:1.1}, \eqref{stable1}, \eqref{stable2}, \eqref{stable3} and $\beta\geq0$. As $n\to \infty$, we have
\begin{align}
&\hat{\mathbf{E}}^\beta\bigg(\frac{W_n^\beta}{D_n^\beta}\bigg)\sim \frac{1}{\Gamma(1-{\frac{1}{\alpha}})n^{\frac{1}{\alpha}}},\label{E:5.1}\\
&\hat{\mathbf{E}}^\beta\bigg(\Big({\frac{W^\beta_n}{D^\beta_n}}\Big)^2\bigg)\sim
\frac{1}{\big(\Gamma(1-{\frac{1}{\alpha}})\big)^{2}n^{\frac{2}{\alpha}}}. \label{E:5.2}
\end{align}
As a consequence, under $\hat{\mathbf{P}}^\beta$,
\begin{align}
\lim_{n\rightarrow\infty} n^{\frac{1}{\alpha}}\frac{W^\beta_n}{D^\beta_n}=\frac{1}{\Gamma(1-{\frac{1}{\alpha}})} \;\;\;in \; probability. \nonumber
\end{align}

\end{proposition}

\textbf{Proof of $\eqref{E:5.1}$ in Proposition $\ref{P:5.1}$}.

Assume \eqref{C:1.1}. For $\beta\geq0$, by Aidekon and Shi \cite{AS1}, we have
\begin{align}
\frac{W^\beta_n}{D^\beta_n}=\hat{\mathbf{E}}^\beta\bigg(\frac{1}{R_\beta\big(V(\omega^\beta_n)\big)}\Big|\mathcal{F}_n\bigg).
\label{Wbn}
\end{align}
Then by \eqref{E:3.2},
\begin{align}
\hat{\mathbf{E}}^\beta\Big(\frac{W_n^\beta}{D_n^\beta}\Big)=
\hat{\mathbf{E}}^\beta\Big(\frac{1}{R_\beta\big(V(\omega^\beta_n)\big)}\Big)=\frac{\mathbf{P}(\underline{S}_n\geq-\beta)}{R(\beta)}.
\nonumber
\end{align}
  From \eqref{E:26}, we know that $\mathbf{P}(\underline{S}_n\geq-\beta)\sim\frac{R(\beta)}{n^{\frac{1}{\alpha}}\Gamma(1-\frac{1}{\alpha})}$, which completes the proof.\qed

\textbf{Proof of $\eqref{E:5.2}$ in Proposition $\ref{P:5.1}$}.

Let $E_n$ be an event such that $\hat{\mathbf{P}}^\beta(E_n)\rightarrow1$ as $n\rightarrow\infty$. We define that
\begin{align}
\xi_{n,E_n^c}:=\hat{\mathbf{E}}^\beta\bigg(\frac{\mathbf{1}_{E_n^c}}{R_\beta\big(V(\omega^\beta_n)\big)}\Big|\mathcal{F}_n\bigg).
\nonumber
\end{align}
By (\ref{Wbn}), we have
\begin{align}
\frac{W^\beta_n}{D^\beta_n}=\hat{\mathbf{E}}^\beta\bigg(\frac{1}{R_\beta\big(V(\omega^\beta_n)\big)}\Big|\mathcal{F}_n\bigg)
=\xi_{n,E_n^c}+\hat{\mathbf{E}}^\beta\bigg(\frac{\mathbf{1}_{E_n}}{R_\beta\big(V(\omega^\beta_n)\big)}\Big|\mathcal{F}_n\bigg),
\nonumber
\end{align}
and
\begin{align}
\hat{\mathbf{E}}^\beta\bigg(\Big({\frac{W^\beta_n}{D^\beta_n}}\Big)^2\bigg)
=\hat{\mathbf{E}}^\beta\bigg({\frac{W^\beta_n}{D^\beta_n}}\cdot\xi_{n,E_n^c}\bigg)
+\hat{\mathbf{E}}^\beta\bigg({\frac{W^\beta_n}{D^\beta_n}}\cdot\frac{\mathbf{1}_{E_n}}{R_\beta\big(V(\omega^\beta_n)\big)}\bigg).\label{E:5.3}
\end{align}
By the H\"{o}lder inequality, we have
\begin{align}
\hat{\mathbf{E}}^\beta\bigg(\Big({\frac{W^\beta_n}{D^\beta_n}}\Big)\cdot\xi_{n,E_n^c}\bigg)
&\leq\bigg[\hat{\mathbf{E}}^\beta\bigg(\Big({\frac{W^\beta_n}{D^\beta_n}}\Big)^2\bigg)\bigg]^{\frac{1}{2}}
\cdot\bigg[\hat{\mathbf{E}}^\beta\Big({\xi^2_{n,E_n^c}}\Big)\bigg]^\frac{1}{2}\nonumber\\
&\leq o(n^{-\frac{2}{\alpha}}).\label{E:5.4}
\end{align}
The last line above is from the following Lemma \ref{L:5.3} and Lemma \ref{L:5.4}.

By \eqref{E:5.3}, \eqref{E:5.4} and the following Lemma \ref{L:5.6}, we have
\begin{align}
\hat{\mathbf{E}}^\beta\bigg(\Big({\frac{W^\beta_n}{D^\beta_n}}\Big)^2\bigg)\leq\frac{1}{(\Gamma(1-\frac{1}{\alpha}))^2n^{\frac{2}{\alpha}}}+o(n^{-\frac{2}{\alpha}}).\nonumber
\end{align}
On the other hand, from the Jensen's inequality,
\begin{align}
\hat{\mathbf{E}}^\beta\bigg(\Big({\frac{W^\beta_n}{D^\beta_n}}\Big)^2\bigg)\geq \bigg(\hat{\mathbf{E}}^\beta\Big({\frac{W^\beta_n}{D^\beta_n}}\Big)\bigg)^2\sim
\frac{1}{(\Gamma(1-\frac{1}{\alpha}))^2n^{\frac{2}{\alpha}}}.\nonumber
\end{align}
The above two inequalities lead to $\eqref{E:5.2}$.
\qed

\begin{lemma}\label{L:5.3}
Assume \eqref{C:1.1}, \eqref{stable1}, \eqref{stable2}, \eqref{stable3} and $\beta\geq0$. We have
\begin{align}
\hat{\mathbf{E}}^\beta\bigg(\Big({\frac{W^\beta_n}{D^\beta_n}}\Big)^2\bigg)
\leq c_{28}n^{-\frac{2}{\alpha}}. \nonumber
\end{align}
\end{lemma}
\begin{proof}
By (\ref{Wbn}) and Jensen's inequality,
\begin{align}
\hat{\mathbf{E}}^\beta\bigg(\Big({\frac{W^\beta_n}{D^\beta_n}}\Big)^2\bigg)
\leq \hat{\mathbf{E}}^\beta\bigg[\frac{1}{\Big(R_\beta\big(V(\omega^\beta_n)\big)\Big)^2}\bigg]
=\frac{1}{R(\beta)}\mathbf{E}\bigg[\frac{\mathbf{1}_{\{\underline{S}_n\geq-\beta\}}}{\big(R_\beta(S_n)\big)}\bigg].\nonumber
\end{align}
Since $R(u)\geq c_1(1+u)$ for any $u\geq0$,
\begin{align}
R(\beta)c_1\hat{\mathbf{E}}^\beta\bigg(\Big({\frac{W^\beta_n}{D^\beta_n}}\Big)^2\bigg)
&\leq\mathbf{E}\bigg[\frac{\mathbf{1}_{\{\underline{S}_n\geq-\beta\}}}{\big(1+\beta+S_n\big)}\bigg]\nonumber\\
&\leq \sum_{i=0}^{{\llcorner n^{\frac{1}{\alpha}}\lrcorner}-1}\mathbf{E}\Big(\frac{\mathbf{1}_{\{-\beta+i\leq S_n<-\beta+i+1\,,\,\underline{S}_n\geq-\beta\}}}{{\big(1+\beta+S_n\big)}}\Big)
+\mathbf{E}\Big(\frac{\mathbf{1}_{\{S_n\geq-\beta+{\llcorner n^{\frac{1}{\alpha}}\lrcorner}\,,\,\underline{S}_n\geq-\beta\}}}{{\big(1+\beta+S_n\big)}}\Big)\nonumber
\end{align}
By Lemma \ref{L:2.2},
\begin{align}
R(\beta)c_1\hat{\mathbf{E}}^\beta\bigg(\Big({\frac{W^\beta_n}{D^\beta_n}}\Big)^2\bigg)
&\leq\sum_{i=0}^{{\llcorner n^{\frac{1}{\alpha}}\lrcorner}-1} c_7\cdot\frac{1}{(i+1)}\cdot\frac{(1+\beta)(i+1)^{\alpha-1}}{n^{1+\frac{1}{\alpha}}}
+\frac{\mathbf{P}(\underline{S}_n\geq-\beta)}{({\llcorner n^{\frac{1}{\alpha}}\lrcorner})}\nonumber\\
&\leq c_{29}n^{-\frac{2}{\alpha}}.\nonumber
\end{align}
\qed
\end{proof}

\begin{lemma}\label{L:5.4}
Assume \eqref{C:1.1}, \eqref{stable1}, \eqref{stable2}, \eqref{stable3} and $\beta\geq0$. For any sequence of events $(E_n)$ such that
$\hat{\mathbf{P}}^\beta(E_n)\rightarrow1$, we have
\begin{align}
\hat{\mathbf{E}}^\beta\Big(\xi_{n,E_n^c}^2\Big)\leq o(n^{-\frac{2}{\alpha}}). \nonumber
\end{align}
\end{lemma}

\begin{proof}
By Jensen's inequality,
\begin{align}
\hat{\mathbf{E}}^\beta\Big(\xi_{n,E_n^c}^2\Big)\leq \hat{\mathbf{E}}^\beta\bigg(\frac{\mathbf{1}_{E_n^c}}{R_\beta\big(V(\omega^\beta_n)\big)^2}\bigg) .\nonumber
\end{align}
For any $\varepsilon>0$,
\begin{align}
\hat{\mathbf{E}}^\beta\Big(\xi_{n,E_n^c}^2\Big)&\leq
\hat{\mathbf{E}}^\beta\bigg(\frac{\mathbf{1}_{E_n^c}}{R_\beta\big(V(\omega^\beta_n)\big)^2}\cdot\mathbf{1}_{\{V(\omega^\beta_n)\geq\varepsilon n^{\frac{1}{\alpha}}\}}\bigg)+\hat{\mathbf{E}}^\beta\bigg(\frac{\,\mathbf{1}_{\{V(\omega^\beta_n)<\varepsilon n^{\frac{1}{\alpha}}\}}\,}{\,R_\beta\big(V(\omega^\beta_n)\big)^2\,} \bigg) \nonumber \\
&=\hat{\mathbf{E}}^\beta\bigg(\frac{\mathbf{1}_{E_n^c}}{R_\beta\big(V(\omega^\beta_n)\big)^2}\cdot\mathbf{1}_{\{V(\omega^\beta_n)\geq\varepsilon n^{\frac{1}{\alpha}}\}}\bigg)+\frac{1}{R(\beta)}\mathbf{E}\bigg(\frac{\,\mathbf{1}_{\{S_n<\varepsilon n^{\frac{1}{\alpha}}\}}\,}{\,\;R_\beta\big(S_n\big)\,}\cdot\mathbf{1}_{\{\underline{S}_n\geq-\beta\}} \bigg).\nonumber
\end{align}
Recall that $R(u)\geq c_1(1+u)$ for any $u\geq0$, so
\begin{align}
\hat{\mathbf{E}}^\beta\Big(\xi_{n,E_n^c}^2\Big)&\leq c_{30}\cdot\frac{\hat{\mathbf{P}}^\beta({E_n}^c)}{(1+\varepsilon n^{\frac{1}{\alpha}}+\beta)^2}+
c_{31}\cdot\mathbf{E}\bigg(\frac{\,\mathbf{1}_{\{S_n<\varepsilon n^{\frac{1}{\alpha}},\;\underline{S}_n\geq-\beta\}}\,}{\,\;S_n+\beta+1\,}\bigg) \nonumber\\
&=o(n^{-\frac{2}{\alpha}})+c_{31}\cdot\mathbf{E}\bigg(\frac{\,\mathbf{1}_{\{S_n<\varepsilon n^{\frac{1}{\alpha}},\;\underline{S}_n\geq-\beta\}}\,}{\,\;S_n+\beta+1\,}\bigg) \nonumber
\end{align}
where the last line is by the assumption $\hat{\mathbf{P}}^\beta(E_n)\rightarrow1$. We observe that
\begin{align}
\mathbf{E}\bigg(\frac{\,\mathbf{1}_{\{S_n<\varepsilon n^{\frac{1}{\alpha}},\;\underline{S}_n\geq-\beta\}}\,}{\,\;S_n+\beta+1}\bigg)&\leq
\sum_{i=0}^{\ulcorner\beta+\varepsilon n^{\frac{1}{\alpha}}\urcorner-1}\mathbf{E}\bigg(\frac{\,\mathbf{1}_{\{-\beta+i\leq S_n<-\beta+i+1,\;\underline{S}_n\geq-\beta\}}\,}{\,\;S_n+\beta+1\,}\bigg) \nonumber \\
&\leq\sum_{i=0}^{\ulcorner\beta+\varepsilon n^{\frac{1}{\alpha}}\urcorner-1}\frac{1}{\;i+1}\mathbf{P}\big(-\beta+i\leq S_n<-\beta+i+1,\;\underline{S}_n\geq-\beta\,\big) \nonumber\\
&\leq \sum_{i=0}^{\ulcorner\beta+\varepsilon n^{\frac{1}{\alpha}}\urcorner-1}\frac{1}{\;i+1}
\cdot\frac{(1+\beta)(1+i)^{\alpha-1}}{n^{1+\frac{1}{\alpha}}},\nonumber
\end{align}
where the last line follows from Lemma \ref{L:2.2}. Then we get
\begin{align}
\mathbf{E}\bigg(\frac{\,\mathbf{1}_{\{S_n<\varepsilon n^{\frac{1}{\alpha}},\;\underline{S}_n\geq-\beta\}}\,}{\,\;(S_n+\beta+1)^{\alpha-1}\,}\bigg)
\leq c_{32}\cdot\varepsilon^{\alpha-1}n^{-\frac{2}{\alpha}}.\nonumber
\end{align}
Therefore,
\begin{align}
\hat{\mathbf{E}}^\beta\Big(\xi_{n,E_n^c}^2\Big)\leq o(n^{-\frac{2}{\alpha}})+c_{33}\cdot\varepsilon^{\alpha-1}n^{-\frac{2}{\alpha}}.\nonumber
\end{align}
Letting $\varepsilon\to 0$, we obtain the desired result.
\qed
\end{proof}

Recall that $\Omega(\omega_i^\beta)$ stands for the set of \lq\lq brothers" of $\omega_i^\beta$. We can write
\begin{align}
W_n^{\beta}=e^{-V(\omega^\beta_n)}\mathbf{1}_{\{\underline{V}(\omega^\beta_n)\geq-\beta\}}+\sum_{i=0}^{n-1}
\sum_{y\in\Omega(\omega_{i+1}^\beta)}\sum_{|x|=n,\,x\geq y}e^{-V(x)}\mathbf{1}_{\{\underline{V}(x)\geq-\beta\}}.\nonumber
\end{align}
Let $k_n<n$ such that $k_n\rightarrow\infty$ as $n\rightarrow\infty$. We denote
\begin{align}
&W_n^{\beta,[0,k_n)}:=\sum_{i=0}^{k_n-1}\sum_{y\in\Omega(\omega_{i+1}^\beta)}\sum_{|x|=n,\,x\geq y}e^{-V(x)}\mathbf{1}_{\{\underline{V}(x)\geq-\beta\}},\nonumber\\
&W_n^{\beta,[k_n,n]}=e^{-V(\omega^\beta_n)}\mathbf{1}_{\{\underline{V}(\omega^\beta_n)\geq-\beta\}}+\sum_{i=k_n}^{n-1}
\sum_{y\in\Omega(\omega_{i+1}^\beta)}\sum_{|x|=n,\,x\geq y}e^{-V(x)}\mathbf{1}_{\{\underline{V}(x)\geq-\beta\}} ,\nonumber
\end{align}
so that $W_n^\beta=W_n^{\beta,\,[0,k_n)}+W_n^{\beta,\,[k_n,n]}$. We define $D_n^{\beta,[0,k_n)}$ and $D_n^{\beta,[k_n,n]}$ similarly. For any $r\in(0,1-\frac{1}{\alpha})$ and $\kappa>\frac{2}{\alpha}$, let
\begin{align}
&E_{n,1}:=\{k_n^r\leq V(\omega^\beta_{k_n})\leq k_n\}\bigcap_{i=k_n+1}^n\{V(\omega^\beta_i)\geq {k_n}^{r/2}\},\nonumber\\
&E_{n,2}:=\bigcap_{i=k_n}^{n-1}\bigg\{\sum_{y\in\Omega(\omega^\beta_{i+1})}\Big[1+\big(V(y)-V(\omega^\beta_i)\big)_+\Big]e^{-(V(y)-V(\omega^\beta_i))}\leq e^{v(\omega^\beta_i)}/2\bigg\},\nonumber\\
&E_{n,3}:=\bigg\{D_n^{\beta,[k_n,n]}\leq\frac{1}{n^\kappa}\bigg\} .\nonumber
\end{align}
We choose
\begin{align}
E_n:=E_{n,1}\cap E_{n,2}\cap E_{n,3}.\label{En}
\end{align}

\begin{lemma}\label{L:5.5}
Assume \eqref{C:1.1}, \eqref{stable1}, \eqref{stable2}, \eqref{stable3} and $\beta\geq0$. Suppose $k_n\to \infty$ satisfying $\frac{k_n^{r/2}}{\log n}\rightarrow\infty$ and  $\frac{k_n}{n^{1/2}}\rightarrow 0$, $n\rightarrow\infty$. Then
\begin{align}
\lim_{n\rightarrow\infty}\hat{\mathbf{P}}^\beta(E_n)=1, \;\;\; \lim_{n\rightarrow\infty}\inf_{u\in[k_n^r,k_n]} \hat{\mathbf{P}}^\beta(E_{n}|V(\omega^\beta_{k_n})=u)=1. \nonumber
\end{align}
\end{lemma}
\begin{proof}
We first prove $\lim_{n\rightarrow\infty}\hat{\mathbf{P}}^\beta(E_n)=1$. We shall check that $\lim_{n\rightarrow\infty}\hat{\mathbf{P}}^\beta(E_{n,l})=1$ for $l=1,2$, and  $\lim_{n\rightarrow\infty}\hat{\mathbf{P}}^\beta(E_{n,\,l}\cap E_{n,\,2}\cap E_{n,\,3}^c)=0$.\\
For $E_{n,1}$: This follows from Lemma \ref{L:4.3}.\\
For $E_{n,2}$: With Lemmas \ref{L:2.4} and  \ref{L:4.2} in hand, one could easily extend Aidekon and Shi's argument to our setting; see \cite[P.18-19]{AS1}.\\
For $E_{n,3}$: Let $\tilde{\mathcal{G}}_\infty:=\sigma\{V(\omega^\beta_k),V(x),x\in\Omega(\omega^\beta_{k+1}),k\geq0\}$ be the $\sigma$-algebra generated by the positions of the spine and its brothers. Then
\begin{align}
\hat{\mathbf{E}}^\beta\big(D_n^{\beta,[k_n,n]}\big|\tilde{\mathcal{G}}_\infty\big)=R_\beta(V(\omega^\beta_n))e^{-V(\omega^\beta_n)}
+\sum_{i=k_n}^{n-1}\sum_{x\in\Omega(\omega^\beta_{i+1})}R_\beta(V(x))e^{-V(x)} \nonumber
\end{align}
holds. For any $x$ on the tree, we have $R_\beta(V(x))\leq c_2(1+\beta+V(\omega^\beta_i))(1+(V(x)-V(\omega^\beta_i))_+)$. Therefore,
\begin{align}
\mathbf{1}_{E_{n,\,1}\cap E_{n,\,2}}\hat{\mathbf{E}}^\beta\big[D_n^{\beta,[k_n,n]}\big|\tilde{\mathcal{G}}_\infty\big]=O(ne^{-\frac{\;k_n^{r/2}}{3}}) \;,\;\;n\rightarrow\infty. \label{E:5.5}
\end{align}
Since $\frac{k_n^{r/2}}{(\log n)}\rightarrow\infty$, and by the Markov inequality we deduce that $\lim_{n\rightarrow\infty}\hat{\mathbf{P}}^\beta(E_{n,\,1}\cap E_{n,\,2}\cap E_{n,\,3}^c)=0$.

It remains to check that $\hat{\mathbf{P}}^\beta(E_n|V(\omega^\beta_{k_n})=u)\rightarrow1$ uniformly in $u\in[k_n^r,k_n]$. Similarly to the proof of Aidekon and Shi \cite{AS1}, we know that $\hat{\mathbf{P}}^\beta(E_{n,2}^c|V(\omega^\beta_{k_n})=u)\rightarrow0$ uniformly in $u\in[k_n^r,k_n]$.
According to \eqref{E:5.5}, $\hat{\mathbf{P}}^\beta(E_{n,3}^c|V(\omega^\beta_{k_n})=u)\rightarrow0$ uniformly in $u\in[k_n^r,k_n]$. Therefore, we only need to check that $\hat{\mathbf{P}}^\beta(E_{n,1}|V(\omega^\beta_{k_n})=u)\rightarrow1$ uniformly in $u\in[k_n^r,k_n]$.
From \eqref{E:3.2}, we have
\begin{align}
\hat{\mathbf{P}}^\beta(E_{n,1}|V(\omega^\beta_{k_n})=u)=\frac{1}{R_\beta(u)}\mathbf{E}(R_\beta(S_{n-k_n}+u)
\mathbf{1}_{\{\underline{S}_{n-k_n}+u\geq{k_n}^{r/2}\}}).\nonumber
\end{align}
Recalling that $\lim_{t\to \infty} R_\beta(t)/t=\theta$. Let $\eta\in(0,\theta)$, and $f_\eta(t):=(\theta-\eta)\min{\{t,\frac{1}{\eta}\}}$. Then $R_\beta(t)\geq bf_\eta(\frac{t}{b})$ for all sufficiently large $t$ and uniformly in $b>0$. Here we take $b:=(n-k_n)^{1/\alpha}$. Hence for
$u\in[k_n^r,k_n]$, we uniformly have
\begin{align}
\hat{\mathbf{P}}^\beta(E_{n,1}|V(\omega^\beta_{k_n})=u)&\geq \frac{(n-k_n)^{1/\alpha}}{R_\beta(u)}
\mathbf{E}\Big(f_\eta(\frac{S_{n-k_n}+u}{(n-k_n)^{1/\alpha}})\mathbf{1}_{\{\underline{S}_{n-k_n}\geq{k_n}^{r/2}-u\}}\Big)\nonumber\\
&\geq \frac{(n-k_n)^{1/\alpha}}{R_\beta(u)}
\mathbf{E}\Big(f_\eta(\frac{S_{n-k_n}+u-k_n^{r/2}}{(n-k_n)^{1/\alpha}})\mathbf{1}_{\{\underline{S}_{n-k_n}\geq{k_n}^{r/2}-u\}}\Big)
.\nonumber
\end{align}
By the assumption $\frac{k_n}{n^{1/2}}\rightarrow 0$, we have $n-k_n\rightarrow\infty$. Hence, by Lemma \ref{L:2.5}, as $n\to \infty$,
\begin{align}
\mathbf{E}\Big(f_\eta(\frac{S_{n-k_n}+u-k_n^{r/2}}{(n-k_n)^{1/\alpha}})\mathbf{1}_{\{\underline{S}_{n-k_n}\geq{k_n}^{r/2}-u\}}\Big)
\sim \frac{R(u-k_n^{r/2})}{\Gamma(1-\frac{1}{\alpha}){(n-k_n)}^{1/\alpha}}\int_0^\infty f_\eta(t)p_\alpha(t)dt \nonumber
\end{align}
holds uniformly in $u\in[k_n^r,k_n]$. Consequently,
\begin{align}
\lim_{n\rightarrow\infty}\inf_{u\in[k_n^r,k_n]} \hat{\mathbf{P}}^\beta(E_{n,1}|V(\omega^\beta_{k_n})=u)\geq
\frac{1}{\Gamma(1-1/\alpha)}\int_0^\infty f_\eta(t)p_\alpha(t)dt.\nonumber
\end{align}
Now note that
\begin{align}
\int_0^\infty f_\eta(t)p_\alpha(t)dt\geq(\theta-\eta)\int_0^{1/\eta}tp_\alpha(t)dt. \nonumber
\end{align}
As $\eta\rightarrow0$, the right side goes to $\theta\mathbf{E}M_\alpha$. Hence by Lemma \ref{L:2.5},
\begin{align}
\lim_{n\rightarrow\infty}\inf_{u\in[k_n^r,k_n]} \hat{\mathbf{P}}^\beta(E_{n,1}|V(\omega^\beta_{k_n})=u)\geq\frac{\theta\mathbf{E}M_\alpha}{\Gamma(1-1/\alpha)}=1.\nonumber
\end{align}
We complete the proof.
\end{proof}
\qed

\begin{lemma}\label{L:5.6}
Assume \eqref{C:1.1}, \eqref{stable1}, \eqref{stable2}, \eqref{stable3} and $\beta\geq0$. $E_n$ is defined as \eqref{En}. Then
\begin{align}
\hat{\mathbf{E}}^\beta\bigg({\frac{W^\beta_n}{D^\beta_n}}\cdot\frac{\mathbf{1}_{E_n}}{R_\beta\big(V(\omega^\beta_n)\big)}\bigg)
\leq \frac{1}{(\Gamma(1-\frac{1}{\alpha}))^2n^{\frac{2}{\alpha}}}+o(n^{-\frac{2}{\alpha}}).\nonumber
\end{align}
\end{lemma}
\begin{proof}
Let ${k_n}$ be a sequence that satisfies $\frac{k_n^{r/2}}{(\log n)}\rightarrow\infty$ and $\frac{k_n}{n^{1/2}}\rightarrow 0$ as $n\rightarrow\infty$. On the set $E_n$, we have $W_n^{\beta,[k_n,n]}\leq D_n^{\beta,[k_n,n]}\leq \frac{1}{n^{\kappa}}$. Observing that $R_\beta\big(V(\omega^\beta_n)\big)\geq1$, so
\begin{align}
\hat{\mathbf{E}}^\beta\bigg({\frac{W_n^{\beta,[k_n,n]}}
{D^\beta_n}}\cdot\frac{\mathbf{1}_{E_n}}{R_\beta\big(V(\omega^\beta_n)\big)}\bigg)
\leq \hat{\mathbf{E}}^\beta\Big(\frac{1}{ n^{\kappa}}\Big)=o(n^{-\frac{2}{\alpha}})\label{E:5.6}
\end{align}
where the last equation from the assumption $\kappa>\frac{2}{\alpha}$.

It remains to treat $\hat{\mathbf{E}}^\beta\Big({\frac{W_n^{\beta,[0,k_n)}}
{D^\beta_n}}\frac{\mathbf{1}_{E_n}}{R_\beta\big(V(\omega^\beta_n)\big)}\Big)$. Since $D^\beta_n\geq D_n^{\beta,[0,k_n)}$,
\begin{align}
\hat{\mathbf{E}}^\beta&\bigg({\frac{W_n^{\beta,[0,k_n)}}
{D^\beta_n}}\cdot\frac{\mathbf{1}_{E_n}}{R_\beta\big(V(\omega^\beta_n)\big)}\bigg)\nonumber\\&\leq
\hat{\mathbf{E}}^\beta\bigg({\frac{W_n^{\beta,[0,k_n)}}
{D_n^{\beta,[0,k_n)}}}\cdot\frac{\mathbf{1}_{E_n}}{R_\beta\big(V(\omega^\beta_n)\big)}\bigg) \nonumber\\
&\leq \hat{\mathbf{E}}^\beta\bigg({\frac{W_n^{\beta,[0,k_n)}}
{D_n^{\beta,[0,k_n)}}}\cdot\mathbf{1}_{\{V(\omega^\beta_{k_n})\in[k_n^r,k_n]\}}\bigg)\cdot
\sup_{u\in[k_n^r,k_n]}\hat{\mathbf{E}}_u^\beta\Big(\frac{1}{R_\beta(V(\omega_{n-k_n}^\beta))}\Big).\nonumber
\end{align}
By \eqref{E:3.2}, $\hat{\mathbf{E}}_u^\beta\Big(\frac{1}{R_\beta(V(\omega_{n-k_n}^\beta))}\Big)=
\frac{1}{R_\beta(u)}\mathbf{E}\big(\mathbf{1}_{\{\underline{S}_{n-k_n}\geq-\beta-u\}}\big)$ and we have
\begin{align}
\sup_{u\in[k_n^r,k_n]}\hat{\mathbf{E}}_u^\beta\Big(\frac{1}{R_\beta(V(\omega_{n-k_n}^\beta))}\Big)
&=\sup_{u\in[k_n^r,k_n]}\frac{1}{R_\beta(u)}\mathbf{E}\big(\mathbf{1}_{\{\underline{S}_{n-k_n}\geq-\beta-u\}}\big)\nonumber\\
&\sim\frac{1}{\,\Gamma(1-\frac{1}{\alpha})(n-k_n)^\frac{1}{\alpha}}\nonumber\\
&\sim \frac{1}{\,\Gamma(1-\frac{1}{\alpha})n^{\frac{1}{\alpha}}},\quad n\to \infty.\nonumber
\end{align}
Hence,
\begin{align}
\hat{\mathbf{E}}^\beta&\bigg({\frac{W_n^{\beta,[0,k_n)}}
{D^\beta_n}}\cdot\frac{\mathbf{1}_{E_n}}{R_\beta\big(V(\omega^\beta_n)\big)}\bigg)\nonumber\\
&\leq\frac{1+o(1)}{\;\Gamma(1-\frac{1}{\alpha})n^{\frac{1}{\alpha}} }\hat{\mathbf{E}}^\beta\bigg({\frac{W_n^{\beta,[0,k_n)}}
{D_n^{\beta,[0,k_n)}}}\cdot\mathbf{1}_{\{V(\omega^\beta_{k_n})\in[k_n^r,k_n]\}}\bigg). \label{E:5.7}
\end{align}
On the other hand,
\begin{align}
\hat{\mathbf{E}}^\beta\bigg({\frac{W_n^{\beta,[0,k_n)}}
{D_n^{\beta,[0,k_n)}}}\cdot\mathbf{1}_{E_n}\bigg) \geq
\hat{\mathbf{E}}^\beta\bigg({\frac{W_n^{\beta,[0,k_n)}}
{D_n^{\beta,[0,k_n)}}}\cdot\mathbf{1}_{\{V(\omega^\beta_{k_n})\in[k_n^r,k_n]\}}\bigg).\nonumber
\end{align}
By Lemma \ref{L:5.5}, $\lim_{n\rightarrow\infty}\inf_{u\in[k_n^r,k_n]} \hat{\mathbf{P}}^\beta\big(E_{n}|V(\omega^\beta_{k_n})=u\big)=1$. Therefore as $n\rightarrow\infty$,
\begin{align}
\hat{\mathbf{E}}^\beta&\bigg({\frac{W_n^{\beta,[0,k_n)}}
{D_n^{\beta,[0,k_n)}}}\mathbf{1}_{\{V(\omega^\beta_{k_n})\in[k_n^r,k_n]\}}\bigg)\nonumber\\
&\leq(1+o(1))\hat{\mathbf{E}}^\beta\bigg({\frac{W_n^{\beta,[0,k_n)}}
{D_n^{\beta,[0,k_n)}}}\mathbf{1}_{E_n}\bigg)\nonumber\\
&\leq(1+o(1))\bigg[\hat{\mathbf{E}}^\beta\bigg({\frac{W_n^{\beta,[0,k_n)}}
{D_n^{\beta,[0,k_n)}}}\mathbf{1}_{E_n}\mathbf{1}_{\{D^\beta_n>\frac{1}{n}\}}\bigg)
+\hat{\mathbf{P}}^\beta\big(D^\beta_n\leq\frac{1}{n}\big)\bigg].\nonumber
\end{align}
Let $\eta_1\in(0,1)$. By the Markov inequality, we see that $\hat{\mathbf{P}}^\beta\big(D^\beta_n\leq\frac{1}{n}\big)\leq \frac{1}{n}\hat{\mathbf{E}}^\beta(\frac{1}{D^\beta_n})=\frac{1}{nR_\beta(0)}$. On the other hand, we already know that
$D_n^{\beta,[k_n,n]}\mathbf{1}_{E_n}=o(\frac{1}{n})$. Therefore, for all sufficient large $n$, $D_n^{\beta,[k_n,n]}\leq\eta_1D^\beta_n$ on $E_n\cap\{D^\beta_n>\frac{1}{n}\}$. Hence, for large $n$,
\begin{align}
\hat{\mathbf{E}}^\beta&\bigg({\frac{W_n^{\beta,[0,k_n)}}
{D_n^{\beta,[0,k_n)}}}\cdot\mathbf{1}_{E_n}\bigg)\nonumber\\
&\leq\frac{1}{1-\eta_1}\hat{\mathbf{E}}^\beta\bigg({\frac{W_n^{\beta,[0,k_n)}}
{D_n^\beta}}\cdot\mathbf{1}_{E_n}\mathbf{1}_{\{D^\beta_n>\frac{1}{n}\}}\bigg)+\frac{1}{nR_\beta(0)}\nonumber\\
&\leq\frac{1}{1-\eta_1}\hat{\mathbf{E}}^\beta\bigg(\frac{W^\beta_n}{D^\beta_n}\bigg)+\frac{1}{nR_\beta(0)}.\nonumber
\end{align}
Recalling (\ref{E:5.1}), $\hat{\mathbf{E}}^\beta\bigg(\frac{W^\beta_n}{D^\beta_n}\bigg)\sim\frac{1}{\Gamma(1-\frac{1}{\alpha})n^{\frac{1}{\alpha}}}$.
Therefore,
\begin{align}
\limsup_{n\rightarrow\infty}n^{\frac{1}{\alpha}}\hat{\mathbf{E}}^\beta&\bigg({\frac{W_n^{\beta,[0,k_n)}}
{D_n^{\beta,[0,k_n)}}}\mathbf{1}_{\{V(\omega^\beta_{k_n})\in[k_n^r,k_n]\}}\bigg)\leq\frac{1}{\Gamma(1-\frac{1}{\alpha})}\cdot\frac{1}{1-\eta_1}.
\nonumber
\end{align}
Letting $\eta_1\rightarrow0$, together with \eqref{E:5.6} and \eqref{E:5.7}, we obtain the desired result.
\end{proof}
\qed

\section{Proof of Theorem 1.2}

From Proposition \eqref{P:5.1}, for any $0<\varepsilon<1$, we have that
\begin{align}
\hat{\mathbf{P}}^\beta\bigg(\bigg|n^{\frac{1}{\alpha}}\frac{W_n^\beta}{D^\beta_n}-\frac{1}{\Gamma(1-\frac{1}{\alpha})}\bigg|>\frac{\varepsilon}{\Gamma(1-1/\alpha)}\bigg)\rightarrow0,\;\;n\rightarrow\infty,
\nonumber
\end{align}
i.e.,
\begin{align}\label{E:6.1}
\mathbf{E}\left(D^\beta_n\mathbf{1}_{\Big\{\Big|n^{\frac{1}{\alpha}}\frac{W_n^\beta}{D^\beta_n}-\frac{1}{\Gamma(1-\frac{1}{\alpha})}\Big|>\frac{\varepsilon}{\Gamma(1-1/\alpha)}\Big\}}\right)\rightarrow0,\;\;n\rightarrow\infty.
\end{align}
Recall that $\Omega_\beta:=\{V(x)>-\beta,\, \forall\; n\geq0,|x|=n\,\}\cap\{\text{nonextinction}\}$ which increases to an event with probability $1$ as $\beta\rightarrow\infty$ under $\mathbf{P}^*$. Let $\eta>0$. There exists a $k_0$ such that $\mathbf{P}^*(\Omega_{k_0})\geq1-\eta$. By \eqref{E:21}, for any $\;\varepsilon>0$, there exists $N $ such that for any $u>N$,
\begin{align}
\theta(1-\varepsilon)u<R(u)<\theta(1+\varepsilon)u .\nonumber
\end{align}
Fix $\beta\!=\!k_0+N$. On $\Omega_{k_0}$, we have $V(x)+\beta>N$. Then
\begin{align}
\theta(1-\varepsilon)(V(x)+\beta)<R_\beta(V(x))<\theta(1+\varepsilon)(V(x)+\beta).\nonumber
\end{align}
Therefore on $\Omega_{k_0}$,
\begin{align}
\theta(1-\varepsilon)(D_n+\beta W_n)<D^\beta_n<\theta(1+\varepsilon)(D_n+\beta W_n).\nonumber
\end{align}
Noticing that $D_n\rightarrow D_\infty>0,\;\mathbf{P}^*-a.s.$ and $W_n\rightarrow0,\;\mathbf{P}^*-a.s.$ Therefore
$\liminf_{n\rightarrow\infty}D^\beta_n>0\;\mathbf{P}^*-a.s.$ on $\Omega_{k_0}$. Let
\begin{align}
A:=\Big\{\Big|n^{1/\alpha}\frac{W^\beta_n}{D^\beta_n}-\frac{1}{\Gamma(1-\frac{1}{\alpha})}\Big|>\frac{\varepsilon}{\Gamma(1-\frac{1}{\alpha})}\Big\}.\;\;\nonumber
\nonumber
\end{align}
From \eqref{E:6.1}, we have $\lim_{n\rightarrow\infty}\mathbf{P}^*(A\cap\Omega_{k_0})=0$. Define
\begin{align}
A_n\!:=\!\bigg\{n^{\frac{1}{\alpha}}\frac{W_n}{D_n+\beta W_n}>(1+\varepsilon)^2\frac{\theta}{\Gamma(1-\frac{1}{\alpha})}\bigg\}\bigcup\bigg\{n^{\frac{1}{\alpha}}\frac{W_n}{D_n+\beta W_n}<(1-\varepsilon)^2\frac{\theta}{\Gamma(1-\frac{1}{\alpha})}\bigg\}.\nonumber
\end{align}
Clearly, $A_n\cap\Omega_{k_0}\subset A\cap\Omega_{k_0}$. Thus $\mathbf{P}^*(A_n\cap\Omega_{k_0})\rightarrow0$ as $n\rightarrow\infty$. Since $\mathbf{P}^*(\Omega_{k_0})\geq1-\eta$, we arrive at
\begin{align}
\limsup_{n\rightarrow\infty} \mathbf{P}^*(A_n)\leq\eta,\nonumber
\end{align}
which implies $n^{\frac{1}{\alpha}}W_n-\frac{\theta}{\;\Gamma(1-\frac{1}{\alpha})}(D_n+\beta W_n)\rightarrow0$ under $\mathbf{P}^*$. Combining this with $D_n\rightarrow D_\infty>0$ and $W_n\rightarrow0,\;\mathbf{P}^*-a.s.$, we complete the proof. \qed

\section{Proof of Theorem 1.3}

The proof of Theorem \ref{T:1.4} also uses the change of probabilities and spinal decomposition. Actually here we need the well-known change-of-probabilities setting in Lyons \cite{Ly}. With the nonnegative martingale $W_n$, we can define a new probability measure $\mathbf{Q}$ such that for any $n\geq1$,
\begin{align}
\mathbf{Q}|_{\mathcal{F}_n}:=W_n\cdot\mathbf{P}|_{\mathcal{F}_n},
\end{align}
where $\mathbf{Q}$ is defined on $\mathcal{F}_\infty(:=\!\!\vee_{n\geq0}\mathcal{F}_n)$. Let us give a description of the branching random walk under $\mathbf{Q}$. We start from one single particle $\omega_0\!\!:=\!\!\varnothing$, located at
$V(\omega_0)=0$. At time $n+1$, each particle $\upsilon$ in the $n$th generation dies and gives birth to a point process independently distributed as $(V(x),|x|=1)$ under $\mathbf{P}_{V(\upsilon)}$ except one particle $\omega_n$, which dies and produces a point process distributed as $(V(x),|x|=1)$ under $\mathbf{Q}_{V(\omega_n)}$. While $\omega_{n+1}$ is chosen to be $\mu$  among the children of $\omega_n$, proportionally to $e^{-V(\mu)}$. Next we state the following fact about the spinal decomposition.\\
\textbf{Fact 7.1 (Lyons \cite{Ly})}.
Assume \eqref{C:1.1}. For any $|x|=n$, we have,
\begin{align}
\mathbf{Q}(\omega_n=x|\mathcal{F}_n)=\frac{\,e^{-V(x)}}{W_n}.\nonumber
\end{align}

The spine process $(V(\omega_n))_{n\geq0}$ under $\mathbf{Q}$ has the distribution of $(S_n)_{n\geq0}$ (introduced in Section 2) under $\mathbf{P}$.

The spinal decomposition is useful in the following lemma, which is the essential  ingredient of the proof of Theorem \ref{T:1.4}.

\begin{lemma} \label{L:7.1}
Let $c_{18}>0$ be the constant in Lemma \ref{L:2.7}. There exists a constant $c_{34}>0$ such that for all large $n$,
\begin{align}
\mathbf{P}\big(\;\exists\, x:n\leq|x|\leq\alpha n,\frac{1}{\alpha}\log n\leq V(x)\leq\frac{1}{\alpha}\log n+c_{16
}\big)\geq c_{34}.\nonumber
\end{align}
\end{lemma}
\begin{proof}
The proof is an extension of the case $\alpha=2$ in Aidekon and Shi~\cite[Lemma 6.3]{AS1}. The idea is borrowed from \cite{AI}. We fix $n$ and let
\begin{equation}
a_i=a_i(n)=\left\{
\begin{aligned}
&0,  &0\leq i\leq \frac{\alpha}{4}n; \\
&\frac{1}{\alpha}\log n, \;\;&\;\frac{\alpha}{4}n<i\leq\alpha n,\; \\
\end{aligned}
\right.\nonumber
\end{equation}
and for $n<k\leq\alpha n$,
\begin{equation}
b_i^{k}=b_i^{k}(n)=\left\{
\begin{aligned}
&i^{\frac{\gamma}{2}},  &0\leq i\leq \frac{\alpha}{4}n; \\
&(k\!-\!i)^{\frac{\gamma}{2}}, \;&\frac{\alpha}{4}n<i\leq k,\; \\
\end{aligned}
\right.\nonumber
\end{equation}
where $\gamma=\frac{1}{\alpha(\alpha\!+\!1)}$.
We define
\begin{align}
&Z^{(n)}:=\sum^{\alpha n}_{k\!=\!n+1}Z^{(n)}_k,\nonumber\\
&Z^{(n)}_k:=\sharp(E_k\cap F_k),\nonumber
\end{align}
and
\begin{align}
&E_k:=\big\{y:|y|=k,V(y_i)\geq a_i, 0\leq i\leq k,V(y)\leq\frac{1}{\alpha}\log n+c_{18}\big\},\nonumber\\
&F_k:=\big\{y:|y|=k,\sum_{v\in\Omega(y_{i+1})}(1+(V(v)-a_i)_+)e^{-(V(v)-a_i)}\leq c'e^{-b_i^k}, 0\leq i\leq k\!-\!1\big\},\nonumber
\end{align}
where $c'$ is a positive constant which will be determined later.
By the definition of $Z^{(n)}$, it is sufficient to prove that there exists $c_{34}>0$ such that
\begin{align}\label{E:7.4}
\mathbf{P}(Z^{(n)}>0)\geq c_{34}.
\end{align}
We start with the first and second moments of $Z^{(n)}$. By \textbf{Fact 7.1},  for $n<k\leq\alpha n$,
\begin{align}\label{E:7.1}
\mathbf{E}(Z^{(n)}_k)=\mathbf{E}_Q\left(\frac{Z^{(n)}_k}{W_k}\right)=\mathbf{E}_Q\big(e^{V(\omega_k)}\mathbf{1}_{\{\omega_k\in E_k\cap F_k\}}\big).
\end{align}
Hence,
\begin{align}
\mathbf{E}(Z^{(n)}_k)\geq n^{\frac{1}{\alpha}}\mathbf{Q}(\omega_k\in E_k\cap F_k).\nonumber
\end{align}
From Lemma \ref{L:2.6} and Lemma \ref{L:2.7}, we can find $c_{35},\,c_{36}$ such that
\begin{align}\label{E:7.2}
\mathbf{Q}(\omega_k\in E_k)=\mathbf{P}(S_i\geq a_i, 0\leq i\leq k, S_k\leq\frac{1}{\alpha}\log n+c_{18})\in \bigg[\frac{c_{35}}{n^{1\!+\!\frac{1}{\alpha}}},\frac{c_{36}}{n^{1\!+\!\frac{1}{\alpha}}}\bigg].
\end{align}
And by Lemma \ref{L:7.3}, for $\varepsilon>0$, it is possible to choose $c'$ such that for sufficiently large $n$,
\begin{align}
\max_{k:n<k\leq\alpha n}\mathbf{Q}(\omega_k\in E_k,\omega_k\notin F_k)\leq \frac{\varepsilon}{n^{1\!+\!\frac{1}{\alpha}}}.\nonumber
\end{align}
Here we choose $\varepsilon=c_{35}/2$.  It follows that for  $n<k\leq \alpha n$,
\begin{align}
\mathbf{Q}(\omega_k\in E_k,\omega_k\in F_k)\geq\frac{c_{35}}{2\,n^{1\!+\!\frac{1}{\alpha}}}.\nonumber
\end{align}
Hence,
\begin{align}
\mathbf{E}(Z^{(n)})\geq\sum^{\alpha n}_{k=n\!+\!1}n^{\frac{1}{\alpha}}\frac{c_{35}}{2\,n^{1\!+\!\frac{1}{\alpha}}} \geq c_{37}.\nonumber
\end{align}
We next estimate the second moment of $Z^{(n)}$. By the definition,
\begin{align}
\mathbf{E}\big((Z^{(n)})^2\big)=\sum^{\alpha n}_{k=n\!+\!1}\sum^{\alpha n}_{l=n\!+\!1}\mathbf{E}\big(Z^{(n)}_kZ^{(n)}_l\big)\leq2\sum ^{\alpha n}_{k=n\!+\!1}\sum^{k}_{l=n\!+\!1}\mathbf{E}\big(Z^{(n)}_kZ^{(n)}_l\big).\nonumber
\end{align}
Similarly to \eqref{E:7.1}, for $n<l\leq k\leq\alpha n$, we get
\begin{align}
\mathbf{E}\big(Z^{(n)}_kZ^{(n)}_l\big)\leq e^{c_{18}}n^{\frac{1}{\alpha}}\mathbf{E}_Q(Z^{(n)}_l\mathbf{1}_{\{\omega_k\in E_k\cap F_k\}}).\nonumber
\end{align}
Consequently,
\begin{align}
\mathbf{E}\big((Z^{(n)})^2\big)\leq2 e^{c_{18}}n^{\frac{1}{\alpha}}\sum ^{\alpha n}_{k=n\!+\!1}\sum^{k}_{l=n\!+\!1}\mathbf{E}_Q(Z^{(n)}_l\mathbf{1}_{\{\omega_k\in E_k\cap F_k\}}).\nonumber
\end{align}
To estimate $\mathbf{E}_Q(Z^{(n)}_l\mathbf{1}_{\{\omega_k\in E_k\cap F_k\}})$, we define
$Y^{(n)}_l:=\sum_{|x|=l}\mathbf{1}_{\{x\in E_k\}}$ which is larger than $Z^{(n)}_l$. For $n<l\leq\alpha n$,
\begin{align}
Y^{(n)}_l=\mathbf{1}_{\{\omega_l\in E_l\}}+\sum^l_{i=1}\sum_{y\in\Omega(\omega_i)}Y^{(n)}_l(y),\nonumber
\end{align}
where $Y^{(n)}_l(y):=\sharp\{x:|x|=l,x\geq y,x\in E_l\}$. Conditioning on $\mathcal{G}_\infty:=\sigma\{\omega_j,V(\omega_j),\Omega(\omega_j),\\(V(u))_{u\,\in\,\Omega(\omega_j)},j\geq1\}\nonumber$,
\begin{align}
\mathbf{E}_Q\big(Y^{(n)}_l(y)|\mathcal{G}_\infty\big)=\varphi_{i,l}(V(y)),\nonumber
\end{align}
where $\varphi_{i,l}(r)=\mathbf{E}\Big(\sum_{|x|=l\!-\!i}\mathbf{1}_{\{r+V(x_j)\geq a_{j+i},\,0\leq j\leq l\!-\!i,\,r+V(x)\leq \frac{1}{\alpha}\log n+c_{18}\}}\Big)$.
Therefore,
\begin{align}
\mathbf{E}\big((Z^{(n)})^2\big)&\leq2 e^{c_{18}}n^{\frac{1}{\alpha}}\Bigg(\sum ^{\alpha n}_{k=n\!+\!1}\mathbf{Q}
(\omega_k\in E_k\cap F_k)+
\sum ^{\alpha n}_{k=n\!+\!1}\sum^{k\!-\!1}_{l=n\!+\!1}\mathbf{Q}
(\omega_k\in E_k\cap F_k,\omega_l\in E_l)\nonumber\\&+
\sum ^{\alpha n}_{k=n\!+\!1}\sum^{k}_{l=n\!+\!1}\sum_{i=1}^l\mathbf{E}_Q\bigg(\mathbf{1}_{\{\omega_k\in E_k\cap F_k\}}\sum_{y\in\Omega(\omega_i)}\varphi_{i,l}\big(V(y)\big)\bigg)\Bigg)
\nonumber\\&=:2 e^{c_{18}}n^{\frac{1}{\alpha}}(I_1+I_2+I_3).\nonumber
\end{align}
We claim that there exists $c_{38}>0$ such that
\begin{align}\label{E:7.3}
I_i\leq \frac{c_{38}}{n^\frac{1}{\alpha}},\quad i=1,2,3.
\end{align}
So we have $\mathbf{E}\big((Z^{(n)})^2\big)\leq c_{39}$, which leads to  \eqref{E:7.4} by the fact
\begin{align}
\mathbf{P}(Z^{(n)}>0)\geq \frac{\big(\mathbf{E}(Z^{(n)})\big)^2}{\mathbf{E}\big((Z^{(n)})^2\big)}.\nonumber
\end{align}
Now let us return to \eqref{E:7.3}. The case $i=1$ immediately follows from \eqref{E:7.2}.
We now discuss the case $i=2$. By \textbf{Fact 7.1} and the Markov property,
\begin{align}
\mathbf{Q}(\omega_k\in E_k,\omega_l\in E_l)&=\mathbf{P}(S_i\geq a_i,\, 0\leq i\leq k, S_l\leq\frac{1}{\alpha}\log n+c_{18},S_k\leq\frac{1}{\alpha}\log n\!+\!c_{18})\nonumber\\
&=\mathbf{E}\big(\mathbf{1}_{\{S_i\geq a_i, \,0\leq i\leq l,\,S_l\leq\frac{1}{\alpha}\log n\!+\!c_{18}\}}f(S_l)\big),\label{E:7.5}
\end{align}
where $f(y):=\mathbf{P}(S_i+y\geq\frac{1}{\alpha}\log n,0\leq i\leq k\!-\!l, S_{k\!-\!l}\!+\!y\leq \frac{1}{\alpha}\log n\!+\!c_{18})$. By Lemma \ref{L:2.3}, we obtain
\begin{align}
f(y)\leq \frac{c_{40}(1+y-\frac{1}{\alpha}\log n)}{(k\!-\!l)^{1+\frac{1}{\alpha}}}.\nonumber
\end{align}
Substituting above into \eqref{E:7.5} yields that
\begin{align}
\mathbf{Q}(\omega_k\in E_k,\omega_l\in E_l)&\leq \frac{c_{41}}{(k\!-\!l)^{1+\frac{1}{\alpha}}}\mathbf{P}(S_i\geq a_i,0\leq i\leq l, S_l\leq \frac{1}{\alpha}\log n+C)\nonumber\\
&\leq \frac{c_{42}}{(k\!-\!l)^{1+\frac{1}{\alpha}}} \cdot \frac{1}{l^{1+\frac{1}{\alpha}}},\nonumber
\end{align}
where in the last step we use Lemma \ref{L:2.6}. Now we obtain
\begin{align}
I_2\leq \sum ^{\alpha n}_{k=n\!+\!1}\sum^{k\!-\!1}_{l=n\!+\!1}\frac{c_{42}}{(k\!-\!l)^{1+\frac{1}{\alpha}}} \cdot \frac{1}{l^{1+\frac{1}{\alpha}}}\leq \frac{c_{43}}{n^{\frac{1}{\alpha}}}.\nonumber
\end{align}

It remains to check \eqref{E:7.3} for $i=3$. Recalling the definition of $\varphi_{i,l}(r)$, by \textbf{the many-to-one formula}, we have
\begin{align}
\varphi_{i,l}(r)&=\mathbf{E}\Big(e^{S_{l\!-\!i}}\mathbf{1}_{\{r+S_j\geq a_{j+i},\,0\leq j\leq l\!-\!i,\,r+S_{l\!-\!i}\leq \frac{1}{\alpha}\log n\!+\!c_{18}\}}\Big)\nonumber\\
&\leq e^{c_{18}-r}n^{\frac{1}{\alpha}}\mathbf{P}\big(r+S_j\geq a_{j+i},\,0\leq j\leq l\!-\!i,\,r+S_{l\!-\!i}\leq \frac{1}{\alpha}\log n\!+\!c_{18}\big).\nonumber
\end{align}
On the one hand, when $i\leq\frac{\alpha n}{4}$, we only need consider $r\geq0$. By Lemma \ref{L:2.6}, we have
\begin{align}
\varphi_{i,l}(r)\leq n^{\frac{1}{\alpha}}e^{c_{44}-r}\frac{1+r}{n^{1\!+\!\frac{1}{\alpha}}}\leq \frac{c_{45}}{n}e^{\!-\!r}(r\!+\!1).\nonumber
\end{align}
If we write $\mathbf{E}_Q[k,i,l]=\mathbf{E}_Q\big(\mathbf{1}_{\{\omega_k\in E_k\cap F_k\}}\sum_{y\in\Omega(\omega_i)}\varphi_{i,l}\big(V(y)\big)\big)$, then
\begin{align}
\mathbf{E}_Q[k,i,l]\leq \frac{c_{45}}{n}\mathbf{E}_Q\Big(\mathbf{1}_{\{\omega_k\in E_k\cap F_k\}}\sum_{y\in\Omega(\omega_i)}e^{-V(y)}(V(y)_+\!+\!1)\Big).\nonumber
\end{align}
Recalling the definition of $F_k$, we obtain
\begin{align}
\mathbf{E}_Q[k,i,l]\leq \frac{c_{45}c'}{n}e^{\!-\!(i\!-\!1)^{\frac{\gamma}{2}}}\mathbf{Q}(\omega_k\in E_k)\leq c_{46}\frac{e^{\!-\!(i\!-\!1)^{\frac{\gamma}{2}}}}{n^{2+\frac{1}{\alpha}}}.\nonumber
\end{align}
Therefore,
\begin{align}
\sum ^{\alpha n}_{k=n\!+\!1}\sum^{k}_{l=n\!+\!1}\sum_{1\leq i\leq \frac{\alpha n}{4}} \mathbf{E}_Q[k,i,l]
\leq \sum ^{\alpha n}_{k=n\!+\!1}\sum^{k}_{l=n\!+\!1} \sum_{1\leq i\leq \frac{\alpha n}{4}}c_{46}\frac{e^{\!-\!(i\!-\!1)^{\frac{\gamma}{2}}}}{n^{2+\frac{1}{\alpha}}}\leq \frac{c_{47}}{n^{\frac{1}{\alpha}}}.\nonumber
\end{align}
On the other hand, when $\frac{\alpha n}{4}<i\leq l$, we only need consider $r\geq \frac{1}{\alpha}\log n$. Then by Lemma \ref{L:2.3}, we have
\begin{align}
\varphi_{i,l}(r)\leq c_{48}n^{\frac{1}{\alpha}}e^{-r}\frac{\;(1+r-\frac{1}{\alpha}\log n)}{(l-i+1)^{1+\frac{1}{\alpha}}}.\nonumber
\end{align}
Similarly, we obtain
\begin{align}
\mathbf{E}_Q[k,i,l]\leq \frac{c_{49}e^{-(k-i+1)^{\frac{\gamma}{2}}}}{n^{1+\frac{1}{\alpha}}(l-i+1)^{1+\frac{1}{\alpha}}}.\nonumber
\end{align}
As a consequence,
\begin{align}
\sum ^{\alpha n}_{k=n\!+\!1}\sum^{k}_{l=n\!+\!1}\sum_{\frac{\alpha n}{4}< i\leq l} \mathbf{E}_Q[k,i,l]
\leq \sum ^{\alpha n}_{k=n\!+\!1}\sum^{k}_{l=n\!+\!1} \sum_{1\leq i\leq \frac{\alpha n}{4}}\frac{c_{49}e^{-(k-i+1)^{\frac{\gamma}{2}}}}{n^{1+\frac{1}{\alpha}}(l-i+1)^{1+\frac{1}{\alpha}}}\leq \frac{c_{50}}{n^{\frac{1}{\alpha}}}.\nonumber
\end{align}
This completes the proof of \eqref{E:7.3}, and then the lemma is now proved.\qed
\end{proof}
\qed

Let
\begin{equation}
p_i=2b_i^k+a=\left\{
\begin{aligned}
&2i^{\frac{\gamma}{2}},  &0\leq i\leq \frac{\alpha}{4}n; \\
&2(n-i)^{\frac{\gamma}{2}}+a, &\frac{\alpha}{4}n<i\leq\alpha n, \\
\end{aligned}
\right.\nonumber
\end{equation}
where $\gamma=\frac{1}{\alpha(\alpha+1)}$ as before.

{\begin{lemma} \label{L:7.0}
For $\varepsilon>0$, there exists $d>0$ such that for any $u\geq0, a\geq0$ and any integer $n\geq1$,
\begin{align}\label{E:7.01}
\mathbf{P}(\exists\, 0\leq i\leq n: S_i\leq p_i-d, \min_{j\leq n}S_j\geq0, \min_{\frac{\alpha n}{4}<j\leq n}S_j\geq a, S_n\leq a+u)\leq \varepsilon\cdot\frac{(1+u)^\alpha}{n^{1+\frac{1}{\alpha}}}.
\end{align}
\end{lemma}

\begin{proof}
Denoting the left side of the inequality by $\mathbf{P}(E)$, then we have
\begin{align}
\mathbf{P}(E)\leq \sum_{i=0}^n\mathbf{P}(E_i), \nonumber
\end{align}
where $E_i:=\{S_i\leq p_i-d, \min_{j\leq n}S_i\geq0, \min_{\frac{\alpha n}{4}<j\leq n}S_j\geq a, S_n\leq a+u\}$.

We first deal with the case $i\leq \frac{\alpha n}{4}$. Now $p_i=2\,i^{\frac{\gamma}{2}}$. By the Markov property,
\begin{align}
\mathbf{P}(E_i)\leq\mathbf{P}\big(\mathbf{1}_{\{S_i\leq\,2\,i^{\frac{\gamma}{2}},\;\underline{S}\,_i\geq0\}}f_1(S_i)\big),\nonumber
\end{align}
where $f_1(y)=\mathbf{P}_y(\underline{S}_{n-i}\geq0, \;\min_{\frac{\alpha n}{4}-i<j\leq n-i}S_j\geq a, S_{n-i}\leq a+u)\leq c_{51}\frac{(1+y)(1+u)^\alpha}{n^{1+\frac{1}{\alpha}}}$ (by Lemma \ref{L:2.6}). Therefore,
\begin{align}
\mathbf{P}(E_i)&\leq c_{51}\cdot\frac{(1+u)^\alpha}{n^{1+\frac{1}{\alpha}}}\mathbf{E}\big((1+S_i)\mathbf{1}_{\{S_i\leq2\,i^\frac{\gamma}{2},
\;\underline{S}\,_i\geq0\}}\big)\nonumber\\&\leq c_{51}c_{10}\cdot\frac{(1+u)^\alpha}{n^{1+\frac{1}{\alpha}}}\cdot\frac{(1+2\,i^\frac{\gamma}{2})^{\alpha+1}}{i^{1+\frac{1}{\alpha}}},\nonumber
\end{align}
which is from Lemma \ref{L:2.3}. Observing that $\frac{\gamma}{2}(\alpha+1)-1-\frac{1}{\alpha}<-1$, we can choose a constant $K_1$ such that for $k\geq K_1$,
\begin{align}
\sum_{i=k}^{\frac{\alpha n}{4}}\mathbf{P}(E_i)\leq \frac{(1+u)^\alpha}{n^{1+\frac{1}{\alpha}}} \varepsilon.\nonumber
\end{align}

Second, we treat the case $\frac{\alpha n}{4}<i<n$. By the Markov property at time $i$, we have
\begin{align}
\mathbf{P}(E_i)=\mathbf{E}\big(\mathbf{1}_{\{S_i\leq2(n-i)^{\frac{\gamma}{2}}+a, \min_{j\leq i}S_i\geq0,\min_{\frac{\alpha n}{4}<j\leq i}S_j\geq a\}}f_2(S_i)\big),\nonumber
\end{align}
where $f_2(y)=\mathbf{P}_y(\underline{S}_{n-i}\geq a,S_{n-i}\leq a+u)\leq c_{10}
\frac{(1+y-a)(1+u)^\alpha}{(n-i)^{1+\frac{1}{\alpha}}}$ (by Lemma \ref{L:2.3}). Then from Lemma \ref{L:2.6}, we have
\begin{align}
\mathbf{P}(E_i)&\leq c_{10}\cdot\frac{(1+u)^\alpha}{(n-i)^{1+\frac{1}{\alpha}}}\mathbf{E}\Big((1+S_i-a)
\mathbf{1}_{\{S_i\leq2(n-i)^{\frac{\gamma}{2}}+a, \min_{j\leq i}S_i\geq0,\min_{\frac{\alpha n}{4}<j\leq i}S_j\geq a\}}\Big)\nonumber\\&\leq c_{52}\cdot\frac{(1+u)^\alpha}{(n-i)^{1+\frac{1}{\alpha}}}
\frac{(1+(n-i)^{\frac{\gamma}{2}})^{\alpha+1}}{n^{1+\frac{1}{\alpha}}}\nonumber\\&\leq c_{52}\cdot\frac{(1+u)^\alpha}{n^{1+\frac{1}{\alpha}}}\cdot
\frac{(1+(n-i)^{\frac{\gamma}{2}})^{\alpha+1}}{(n-i)^{1+\frac{1}{\alpha}}}.\nonumber
\end{align}
Therefore, we can find $K_2$ such that when $k\geq K_2$,
\begin{align}
\sum_{i=\frac{\alpha n}{4}}^{n-k}\mathbf{P}(E_i)\leq \frac{(1+u)^\alpha}{n^{1+\frac{1}{\alpha}}}\varepsilon.\nonumber
\end{align}
Notice that our choice of $K(:=\max(K_1,K_2))$ does not depend on the constant $d$. Thus, we are allowed to choose $d\geq 2\,K^\frac{\gamma}{2}+a$ such that $\mathbf{P}(E_i)=0$ if $i\in[0,K]\cup[n-K,n].$
\qed
\end{proof}

\begin{lemma}\label{L:7.3}
For $\varepsilon>0$, it is possible to choose $c'$ such that for all large $n$,
\begin{align}
\max_{k:n<k\leq\alpha n}\mathbf{Q}\,(\omega_k\in E_k,\omega_k\notin F_k)\leq \frac{\varepsilon}{n^{1\!+\!\frac{1}{\alpha}}}.\nonumber
\end{align}
$E_k, F_k$ and $c'$ are defined as before.
\end{lemma}

\begin{proof}
By the definition of $F_k$, we can write
\begin{align}
&\mathbf{Q}\,(\omega_k\in E_k,\omega_k\notin F_k)\nonumber\\=&\mathbf{Q}\,(\omega_k\in E_k, \exists\,0\leq i\leq k-1,
\sum_{\nu\in\Omega(\omega_{i+1})}\big(1+(V(\nu)-a_i)_+\big)e^{-(V(\nu)-a_i)}>c'e^{-b_i^k})\nonumber\\
\leq &\mathbf{Q}\,\big(\omega_k\in E_k, \exists\,0\leq i\leq k-1, V(\omega_i)\leq 2b_i^k+a_i-d\big)\nonumber\\&+
\sum_{i=0}^{k-1}\mathbf{Q}\,\bigg(\omega_k\in E_k, \sum_{\nu\in\Omega(\omega_{i+1})}
\big(1+(V(\nu)-a_i)_+\big)e^{-(V(\nu)-a_i)}>c'e^{-(V(\omega_i)-a_i+d)/2}\bigg).\nonumber
\end{align}
By Lemma \ref{L:7.0}, for any $k\in(n,\alpha n]$,
\begin{align}
\mathbf{Q}\,\big(\omega_k\in E_k, \exists\,0\leq i\leq k-1, V(\omega_i)\leq 2b_i^k+a_i-d\big)\leq\frac{\varepsilon}{n^{1+\frac{1}{\alpha}}}.\nonumber
\end{align}
We see that $1+a_+\leq1+b_++(a-b)_+\leq(1+b_+)(1+(a-b)_+)$. Thus,
\begin{align}
&\sum_{\nu\in\Omega(\omega_{i+1})}\big(1+(V(\nu)-a_i)_+\big)e^{-(V(\nu)-a_i)}\nonumber\\
\leq &\big(1+(V(\omega_i)-a_i)_+\big)e^{-(V(\omega_i)-a_i)}\sum_{\nu\in\Omega(\omega_{i+1})}
\big(1+(V(\nu)-V(\omega_i))_+\big)e^{-(V(\nu)-V(\omega_i))}.\nonumber
\end{align}
For convenience, we let $\xi(\omega_{i+1}):=\sum_{\nu\in\Omega(\omega_{i+1})}
\big(1+(V(\nu)-V(\omega_i))_+\big)e^{-(V(\nu)-V(\omega_i))}$. Then we only need to prove that, for $c'$ large enough,
\begin{align}
\sum_{i=0}^{k-1}\mathbf{Q}\,\Big(\xi(\omega_{i+1})>c'\cdot\frac{e^{(V(\omega_i)-a_i)/2}}{1+(V(\omega_i)-a_i)_+},\omega_k\in E_k\Big)\leq \frac{\varepsilon}{n^{1+\frac{1}{\alpha}}},  \;\;\forall\;k\in(n,\alpha n].\nonumber
\end{align}
Actually, it is enough to show that
\begin{align}
\sum_{i=0}^{k-1}\mathbf{Q}\, \Big(\xi(\omega_{i+1})>c'\cdot e^{(V(\omega_i)-a_i)/3},\omega_k\in E_k\Big)\leq \frac{\varepsilon}{n^{1+\frac{1}{\alpha}}},  \;\;\forall\;k\in(n,\alpha n].\nonumber
\end{align}
First, we deal with the case $i+1\leq \frac{\alpha n}{4}$. We notice that
\begin{align}
\mathbf{Q}\,\big(\xi(\omega_{i+1})>c'\cdot e^{(V(\omega_i)-a_i)/3},\omega_k\in E_k\big)\leq \mathbf{Q}\,\big(\xi(\omega_{i+1})>c'e^{V(\omega_i)/3},\omega_k\in E_k\big).\nonumber
\end{align}
By the Markov property at time $i+1$, we get the right side of above is same as
\begin{align}
\mathbf{Q}\,\big(f(V(\omega_{i+1}))\mathbf{1}_{\{\xi(\omega_{i+1})>c'e^{V(\omega_i)/3},V(\omega_j)\geq0,\, j\leq i+1\}}\big),\nonumber
\end{align}
where $f(r)=\mathbf{P}_r(V(\omega_j)\geq a_{j+i+1}, 0\leq j\leq k-i-1, V(\omega_{k-i-1})\leq \frac{1}{\alpha}\log n+c_{18})$. By Lemma \ref{L:2.6}, $f(r)\leq c_{53}\cdot\frac{(1+r)}{n^{1+\frac{1}{\alpha}}}$ when $r\geq0$. This yields that
\begin{align}\label{E:7.03}
\mathbf{Q}(\xi(\omega_{i+1})>c'\cdot e^{V(\omega_i)/3},\omega_k\in E_k)\leq \frac{c_{53}}{n^{1+\frac{1}{\alpha}}}
\mathbf{E}_Q((1+V(\omega_{i+1})_+)\mathbf{1}_{\{\xi(\omega_{i+1})>c'e^{V(\omega_j)/3},V(\omega_j)\geq0,\, j\leq i+1\}}).
\end{align}
On the one hand, we have
\begin{align}
1+V(\omega_{i+1})_+\leq 1+V(\omega_i)_++(V(\omega_{i+1})-V(\omega_i))_+.\nonumber
\end{align}
By the Markov property at time $i$, we obtain that
\begin{align}
\mathbf{E}_Q\big((1+&V(\omega_{i+1})_+)\mathbf{1}_{\{\xi(\omega_{i+1})>c'e^{V(\omega_i)/3},V(\omega_j)\geq0, j\leq i+1)\}}\big)\nonumber\\ &\leq\mathbf{E}_Q(\mathbf{1}_{\{V(\omega_j)\geq0, j\leq i\}}\kappa(V(\omega_i))),\nonumber
\end{align}
where for $x\geq0$, $\kappa(x):=\mathbf{1}_{\{\xi>c'e^{x/3}\}}(1+x+\delta_+)$. $(\xi,\delta)$ is the identical and independent copy of $(\xi(\omega_1),V(\omega_1))$ under $\mathbf{Q}$, and independent of the other random variables. In view of
\eqref{E:7.03}, it follows that
\begin{align}
&\sum_{i=0}^{\frac{\alpha n}{4}-1}\mathbf{Q}(\xi(\omega_{i+1})>c'e^{V(\omega_i)/3},\omega_k\in E_k)\\ \nonumber
&\;\;\leq \frac{c_{53}}{n^{1+\frac{1}{\alpha}}}\sum_{i=0}^\infty\mathbf{E}_Q\Big(\big(1+V(\omega_i)_++\delta_+\big)\mathbf{1}_{\{\min_{j\leq i} V(\omega_j)\geq0,V(\omega_i)\leq 3(\ln\xi-\ln c')\}}\Big).\nonumber
\end{align}
Notice that the term inside the expectation is $0$ if $c'>\xi$. Therefore, by Lemma \ref{L:2.4}, we get that
\begin{align}
&\sum_{i=0}^{\frac{\alpha n}{4}-1}\mathbf{Q}\big(\xi(\omega_{i+1})>c'e^{V(\omega_i)/3},\omega_k\in E_k\big)\\ \nonumber
&\leq \frac{c_{54}}{n^{1+\frac{1}{\alpha}}} \mathbf{E}_Q\big((1+\ln\xi_++\delta_+)(1+\ln\xi_+)^{\alpha-1}\mathbf{1}_{\{c'\leq\xi\}}\big).\nonumber
\end{align}
Observe that $\xi\leq X+\tilde{X}$, going back to the measure $\mathbf{P}$, we get
\begin{align}
\sum_{i=0}^{\frac{\alpha n}{4}-1}\mathbf{Q}(\xi&(\omega_{i+1})>c'e^{V(\omega_i)/3},\omega_k\in E_k) \leq \frac{c_{54}}{n^{1+\frac{1}{\alpha}}}  \bigg(\mathbf{E}\Big(X\mathbf{1}_{\{c'\leq X+\tilde{X}\}}\big(1+\ln_+(X+\tilde{X})\big)^\alpha\Big)\\ &+\mathbf{E}\Big(\tilde{X}\mathbf{1}_{\{c'\leq X+\tilde{X}\}}\big(1+\ln_+(X+\tilde{X})\big)^{\alpha-1}\Big)\bigg)\leq\frac{\varepsilon}{n^{1+\frac{1}{\alpha}}}\nonumber
\end{align}
for $c'$ large enough since $\mathbf{E}(X(1+\ln_+(X+\tilde{X}))^\alpha)+\mathbf{E}(\tilde{X}(1+\ln_+(X+\tilde{X}))^{\alpha-1})
<\infty$.
It remains to treat the case $i+1>\frac{\alpha n}{4}$. We want to show that for $c'$ large enough,
\begin{align}
\sum_{i=\frac{\alpha n}{4}}^{k-1}\mathbf{Q}(\xi(\omega_{i+1})>c'\cdot e^{(V(\omega_i)-a_i)/3},\omega_k\in E_k)\leq \frac{\varepsilon}{n^{1+\frac{1}{\alpha}}}\label{E:7.02}.
\end{align}
We deduce that
\begin{align}
\mathbf{Q}(\xi(\omega_{i+1})>c'\cdot e^{(V(\omega_i)-a_i)/3},\omega_k\in E_k)=\mathbf{Q}(\bar{\xi}(V(\omega_{i+1})-V(\omega_i))>c'\cdot e^{(V(\omega_i)-a_i)/3},\omega_k\in E_k),\nonumber
\end{align}
where, given $(V(\omega_i),i\leq k)$, the random variable $\bar{\xi}(V(\omega_{i+1})-V(\omega_i))\in\mathbb{R}$ has the distribution of $\int_{x\in\mathbb{R}}(1+x_+)e^{-x}\mu(dx)$ under $\mathbf{Q}_{V(\omega_{i+1})-V(\omega_i)}(d\mu)$. The last equation is same as
\begin{align}\label{E:7.05}
\mathbf{P}(\bar{\xi}(S_{i+1}-S_i)>c'\cdot e^{(S_i-a_i)/3},\underline{S}_k\geq0,\min_{\frac{\alpha n}{4}<j\leq k}S_j\geq\frac{1}{\alpha}\log n,S_k\leq \frac{1}{\alpha}\log n+c_{18}).
\end{align}
The random variable $\bar{\xi}(S_{i+1}-S_i)$ has the distribution of $\int_{x\in\mathbb{R}}(1+x_+)e^{-x}\mu(dx)$ under $\mathbf{P}_{S_{i+1}-S_i}(d\mu)$. We return time, that is, we replace $S_i$ by $\hat{S}_k-\hat{S}_{k-i}$. $\{\hat{S}_i\}$  is a random walk identically distributed with $\{S_i\}$. Then  \eqref{E:7.05} changes into
\begin{align}
&\mathbf{P}(\bar{\xi}(\hat{S}_{k-i}-\hat{S}_{k-i-1})>c'e^{(\hat{S}_k-\hat{S}_{k-i}-a_i)/3},{\underline{-\hat{S}}}_k\geq -\hat{S}_k, \nonumber \\&\;\;\;\;\;\;\;\;\;\;\;\min_{0\leq j<k-\frac{\alpha n}{4}}(-\hat{S}_j)\geq
\frac{1}{\alpha}\log n-\hat{S}_k, \frac{1}{\alpha}\log n\leq \hat{S}_k\leq \frac{1}{\alpha}\log n+c_{18})\label{E:7.07}\\ =&\mathbf{P}\bigg(\bar{\xi}(\hat{S}_{k-i}-\hat{S}_{k-i-1})>c'e^{-\hat{S}_{k-i}/3},{\underline{-\hat{S}}}_k\geq-\frac{1}{\alpha}\log n -c_{18},\nonumber\\&\;\;\;\;\;\;\;\;\;\;\; \min_{0\leq j<k-\frac{\alpha n}{4}}(-\hat{S}_j)\geq -c_{18}, \frac{1}{\alpha}\log n\leq \hat{S}_k\leq \frac{1}{\alpha}\log n+c_{18}\bigg).\nonumber
\end{align}
Using the Markov property at time $k-i$, the above probability equals to
\begin{align}
\mathbf{P}\Big( f(-\hat{S}_{k-i})\,;\,\bar{\xi}(\hat{S}_{k-i}-\hat{S}_{k-i-1})>c'e^{-\hat{S}_{k-i}/3},{\underline{-\hat{S}}}_{k-i}\geq -c_{18}\Big),\nonumber
\end{align}
where for $-r\geq -c_{18}$,
\begin{align}
f(-r)=&\mathbf{P}_{-r}(\frac{1}{\alpha}\log n\leq \hat{S}_i\leq \frac{1}{\alpha}\log n+c_{18}, {\underline{-\hat{S}}}_i\geq-\frac{1}{\alpha}\log n -c_{18},\min_{0\leq j<i-\frac{\alpha n}{4}}(-\hat{S}_j)\geq -c_{18})\nonumber\\ \leq& c_{55}\frac{(1+c_{18}-r)^{\alpha-1}}{n^{1+\frac{1}{\alpha}}}.\nonumber
\end{align}
The last line is from Lemma \ref{L:2.6b} in which we treat $\{-\hat{S_n}\}$ as a new random walk with the step  distributed as $-S_1$. After a time reversal, we obtain
\begin{align}
\mathbf{P}\eqref{E:7.07}&\leq \frac{c_{55}}{n^{1+\frac{1}{\alpha}}}\mathbf{E}\bigg((1-\hat{S}_{k-i}+c_{18})^{\alpha-1}\mathbf{1}_{\{\bar{\xi}(\hat{S}_{k-i}-\hat{S}_{k-i-1})>c'e^{-\hat{S}_{k-i}/3},{\underline{-\hat{S}}}_{k-i}\geq -c_{18}\}}\bigg)\nonumber\\ &\leq \frac{c_{55}}{n^{1+\frac{1}{\alpha}}}\mathbf{E}_Q\bigg((1-V(\omega_{k-i})+c_{18})^{\alpha-1}\mathbf{1}_{\{{\xi}(\omega_{k-i})>c'e^{-V(\omega_{k-i})/3},-V(\omega_j)\geq -c_{18},\,j\leq k-i\}}\bigg)\label{E:7.04}.
\end{align}
Similarly to dealing with \eqref{E:7.03}, we obtain by the Markov property at time $k\!-\!i\!-\!1$ (the expectation above we denoted by $\mathbf{E}_Q\eqref{E:7.04})$,
\begin{align}\label{E:7.08}
\mathbf{E}_Q\eqref{E:7.04}\leq \mathbf{E}_Q(\mathbf{1}_{\{-V(\omega_j)\geq -c_{18},\,j\leq k-i-1\}}\lambda(V(\omega_{k-i-1}))),\nonumber
\end{align}
where for $x\geq-c_{18}$, $\lambda(x):=\mathbf{1}_{\{\xi>c'e^{-(x+\delta)/3}\}}(1+(c_{18}-(x+\delta))_+)^{\alpha-1}$. Hence,
\begin{align}
&\sum_{i=\frac{\alpha n}{4}}^{k-1}\mathbf{E}_Q\eqref{E:7.04}\\ \leq \sum_{i=\frac{\alpha n}{4}}^{k-1}&\mathbf{E}_Q\Big(\big(1+(c_{18}-(V(\omega_{k-i-1})+\delta))_+\big)^{\alpha-1}\mathbf{1}_{\{-V(\omega_j)\geq -c_{18},\,j\leq k-i-1,V(\omega_{k-i-1})>3(\ln c'-\ln \xi)-\delta\}}\Big).\nonumber
\end{align}
By Lemma \ref{L:2.10} $(\romannumeral2)$,
\begin{align}
&\sum_{i=\frac{\alpha n}{4}}^{k-1}\mathbf{E}_Q\eqref{E:7.04}\\ \nonumber&\leq c_{56}\mathbf{E}_Q\bigg(\big(1+\ln_+\xi\big)^{\alpha-1}\big(1+\ln_+\xi+\delta_+\big)\mathbf{1}_{\{\ln c'<\ln\xi+c_{18}+\frac{\delta}{3}\}}\bigg)\nonumber \\ &\leq c_{56}\mathbf{E}\bigg(\Big({X}\big(1+\ln_+ (X+\tilde{X})\big)^\alpha+\nonumber\\&{\tilde{X}}\big(1+\ln_+ (X+\tilde{X})\big)^{\alpha-1}\Big)\mathbf{1}_{\{\ln c<\ln(X+\tilde{X})+c_{18}+\frac{1}{3}\max_{|x|=1}V(x)\}}\bigg)\nonumber\\&\leq
\epsilon\nonumber
\end{align}
for $c'$ large enough since $\mathbf{E}(X(1+\ln_+(X+\tilde{X}))^\alpha)+\mathbf{E}(\tilde{X}(1+\ln_+(X+\tilde{X}))^{\alpha-1})
<\infty$.

\end{proof}

We now can use Lemma \ref{L:7.1} to prove the following theorem which is much closer to Theorem~\ref{T:1.4}.
\begin{theorem}\label{T:7.2}
\begin{align}
{\underline{\lim}}_{n\rightarrow\infty}\Big(\min_{|x|=n}V(x)-\frac{1}{\alpha}\log n\Big)=-\infty, \;\;\mathbf{P}^*\!\!-\!\!a.s.\nonumber
\end{align}
\end{theorem}

\begin{proof}
This proof is similar to the case $\alpha=2$ in Aidekon and Shi \cite{AS1}. For completeness, we still present it here. By our assumption, we have $\mathbf{P}(\min_{|x|=1}V(x)<0)>0$. Thus for any $J>0$, there exists an integer $L=L(J)\geq1$ such that
\begin{align}
c_{57}:=\mathbf{P}(\min_{|x|=L}V(x)\leq -J)>0.\nonumber
\end{align}
Let $n_k:=(L+\alpha)^k,k\geq1$, which means $n_{k+1}\geq\alpha n_k+L$. And let
\begin{align}
&T_k:=\inf\{i\geq n_k:\min_{|x|=i}V(x)\leq \frac{1}{\alpha}\log n_k\!+\!c_{18}\},\nonumber\\
&G_k:=\{T_k\leq \alpha n_k\}\cap\{\min_{|y|=L}\big[V(x_{(k)}y)-V(x_{(k)})\big]\leq-J\},\nonumber
\end{align}
where $x_{(k)}$ is the individual satisfying $|x_{(k)}|=T_k, V(x_{(k)})\leq \frac{1}{\alpha}\log n_k\!+\!c_{18}$ (when $T_k\leq \alpha n_k$) and $x_{(k)}y$ denotes the connection of $x_{(k)}$ with $y$. More precisely, $x_{(k)}y$ denotes a vertex $z$ on the tree such that $x_{(k)}$ is the ancestor of it and $|z|=|x_{(k)}|\!+\!|y|$.
For any pair of positive integers $j\leq l$,
\begin{align}\label{E:7.6}
\mathbf{P}\big(\bigcup_{k=j}^lG_k\big)=\mathbf{P}\big(\bigcup_{k=j}^{l-1}G_k\big)+
\mathbf{P}\big(\bigcap_{k=j}^{l-1}G^c_k\cap G_l\big).
\end{align}
By the Markov property at time $T_l$,
\begin{align}
\mathbf{P}(G_l|\mathcal{F}_{T_l})=\mathbf{1}_{\{T_l\leq\;\alpha n_l\}}\;\mathbf{P}(\min_{|y|=L}V(y)\leq-J)=c_{57}\mathbf{1}_{\{T_l\leq\;\alpha n_l\}}.\nonumber
\end{align}
Since $\bigcap_{k=j}^{l-1}G^c_k$ is $\mathcal{F}_{T_l}-$measurable,
\begin{align}
\mathbf{P}\big(\bigcap_{k=j}^{l-1}G^c_k\cap G_l\big)&=c_{57}\mathbf{E}\big(\mathbf{1}_{\{\bigcap_{k=j}^{l-1}G^c_k\}}\cdot\mathbf{1}_{\{T_l\leq\;\alpha n_l\}}\big)\nonumber\\
&\geq c_{57}\mathbf{P}(T_l\leq\alpha n_l)-c_{57}\mathbf{P}(\bigcup_{k=j}^{l-1}G_k).\nonumber
\end{align}
Lemma \ref{L:7.1} tells us for all large $l$ (say $l\geq j_0$), $\mathbf{P}\{T_l\leq\alpha n_l\}\geq c_{34}$. Combining this with \eqref{E:7.6}, we obtain
\begin{align}
\mathbf{P}\big(\bigcup_{k=j}^lG_k\big)\geq (1-c_{57})\mathbf{P}\big(\bigcup_{k=j}^{l-1}G_k\big)+c_{34}c_{57},\;\; \;j_0\leq j.\nonumber
\end{align}
Iterating the inequality above yields that
\begin{align}
\mathbf{P}\big(\bigcup_{k=j}^lG_k\big)\geq c_{34}c_{57}\sum_{i=0}^{l-j-1}(1-c_{57})^i.\nonumber
\end{align}
Letting $l\to \infty$, we get $\mathbf{P}\big(\bigcup_{k=j}^\infty G_k\big)\geq c_{34}$ ($j\geq j_0$). As a consequence,
\begin{align}
\mathbf{P}({\overline{\lim}}_{k\rightarrow\infty}G_k)\geq c_{34}.\nonumber
\end{align}
On the event ${\overline{\lim}} \;G_k$, there are infinitely many vertices $x$ such that $V(x)\leq \frac{1}{\alpha}\log|x|+c_{18}-J$. Hence,
\begin{align}
\mathbf{P}\Big({\underline{\lim}}_{n\rightarrow\infty}\big(\min_{|x|=n}V(x)-\frac{1}{\alpha}\log n\big)\leq c_{18}-J\Big)\geq c_{34}.\nonumber
\end{align}
By the arbitrary of  $J>0$, we obtain
\begin{align}
\mathbf{P}\Big({\underline{\lim}}_{n\rightarrow\infty}\big(\min_{|x|=n}V(x)-\frac{1}{\alpha}\log n\big)\leq -\infty\Big)\geq c_{34}.\nonumber
\end{align}
Recalling that the system survives almost surely under $\mathbf{P}^*$, and by the branching property, it follows that
\begin{align}
\mathbf{P}^*\Big({\underline{\lim}}_{n\rightarrow\infty}\big(\min_{|x|=n}V(x)-\frac{1}{\alpha}\log n\big)\leq -\infty\Big)=1,\nonumber
\end{align}
which completes the proof. \qed
\end{proof}
\qed\\
Now we are ready to prove Theorem \ref{T:1.4}.\\
\\
\emph{Proof of Theorem \ref{T:1.4}.} Observing that $W_n=\sum_{|x|=n}e^{-V(x)}\geq \exp\{-\min_{|x|=n}V(x)\}$, then
\begin{align}
n^{\frac{1}{\alpha}}W_n\geq \exp\,\big\{\frac{1}{\alpha}\log n-\min_{|x|=n}V(x)\big\}.\nonumber
\end{align}
It follows from Theorem \ref{T:7.2} that
\begin{align}
\limsup_{n\rightarrow\infty} \,n^{\frac{1}{\alpha}}W_n=\infty\;\;\;\mathbf{P}^*\!\!-\!\!a.s. \nonumber
\end{align}
\qed

\noindent{\bf Acknowledgement.} We are grateful to Dr. Xinxin Chen for enlightenment on Lemma 2.11.

\bigskip\bigskip

\bigskip\bigskip

\noindent{\small Laboratory of Mathematics
and Complex Systems, School of Mathematical Sciences, Beijing Normal
University, Beijing 100875, P.R. China}

\noindent{\small Emails: hehui@bnu.edu.cn

\quad\quad\,  liujingning14@163.com

\quad\quad\,  meizhang@bnu.edu.cn}

\end{document}